\documentclass{amsart}
\usepackage{amsmath,amsthm,amssymb}
\usepackage{dsfont}
\usepackage[all,cmtip]{xy}
\input{diagxy}

\usepackage[cm]{fullpage}

 \theoremstyle{plain}
 \newtheorem{thm}{Theorem}[section]
 \newtheorem{cor}[thm]{Corollary}
 \newtheorem{lem}[thm]{Lemma}
 
 \newtheorem{prop}[thm]{Proposition}
 
\theoremstyle{definition}
 \newtheorem{defn}[thm]{Definition}

\theoremstyle{remark}
 \newtheorem{rem}[thm]{Remark}

 \newtheorem{nota}[thm]{Notation}
 \newtheorem{conv}[thm]{Convention}

 \numberwithin{equation}{section}

\DeclareMathOperator{\VF}{VF} \DeclareMathOperator{\ACVF}{ACVF}
 
\DeclareMathOperator{\RV}{RV} 
\DeclareMathOperator{\MM}{\mathcal{M}}
\DeclareMathOperator{\LL}{\mathcal{L}}
\DeclareMathOperator{\OO}{\mathcal{O}}

 \DeclareMathOperator{\ran}{ran}
 \DeclareMathOperator{\dom}{dom}

 \DeclareMathOperator{\id}{id}

 \DeclareMathOperator{\lh}{lh}
 
 \DeclareMathOperator{\spec}{Spec}

 \DeclareMathOperator{\td}{tr\,deg}

 \DeclareMathOperator{\acl}{acl}
 \DeclareMathOperator{\dcl}{dcl}
 \DeclareMathOperator{\pr}{pr}
 \DeclareMathOperator{\alg}{ac}

\DeclareMathOperator{\jcb}{Jcb}

\DeclareMathOperator{\K}{\overline{K}}

\def\XXint#1#2#3{{\setbox0=\hbox{$#1{#2#3}{\int}$}
\vcenter{\hbox{$#2#3$}}\kern-.5\wd0}}

\newcommand{\Z}{\mathds{Z}}

\newcommand{\N}{\mathds{N}}

\newcommand{\PP}{\mathds{P}}
\newcommand{\p}{$p$\nobreakdash}

\newcommand{\cmin}{$C$\nobreakdash}
\newcommand{\vmin}{$V$\nobreakdash}

\newcommand{\gB}{\mathfrak{B}}
\newcommand{\gC}{\mathfrak{C}}

\newcommand{\ga}{\mathfrak{a}}
\newcommand{\gb}{\mathfrak{b}}
\newcommand{\gc}{\mathfrak{c}}

\newcommand{\go}{\mathfrak{o}}
\newcommand{\gp}{\mathfrak{p}}
\newcommand{\gq}{\mathfrak{q}}
\newcommand{\gr}{\mathfrak{r}}

\newcommand{\0}{\emptyset}

 \newcommand{\abs}[1]{\left\vert#1\right\vert}
 
 \newcommand{\set}[1]{\left\{#1\right\}}
 \newcommand{\seq}[1]{\left<#1\right>}

 \newcommand{\lan}[1]{\mathcal{L}_{\textup{#1}}}

\newcommand{\mdl}[1]{\mathcal{#1}}  
\newcommand{\bb}[1]{\mathbb{#1}}

\newcommand{\liff}{\leftrightarrow}
\newcommand{\ex}[1]{\exists #1 \;} 

\newcommand{\rest}{\upharpoonright}
\newcommand{\lbar}{\vec}

\newcommand{\fun}{\longrightarrow}
\newcommand{\efun}{\longmapsto}
\newcommand{\sub}{\subseteq}

\newcommand{\mi}{\smallsetminus}

\newcommand{\la}{\langle}
\newcommand{\ra}{\rangle}

\newtheorem*{incla}{Claim}

\DeclareMathOperator{\mVF}{\mu \! \VF}
\DeclareMathOperator{\mgVF}{\mu_{\Gamma} \! \VF}
\DeclareMathOperator{\mRV}{\mu \! \RV}
\DeclareMathOperator{\mgRV}{\mu_{\Gamma} \! \RV}
\DeclareMathOperator{\rv}{rv}

\DeclareMathOperator{\vv}{val}

\DeclareMathOperator{\mor}{Mor}
\DeclareMathOperator{\gsk}{\mathbf{K}_+}
\DeclareMathOperator{\ggk}{\mathbf{K}}
\DeclareMathOperator{\ob}{Ob}

\DeclareMathOperator{\fib}{fib}

\DeclareMathOperator{\isp}{I_{sp}}

\DeclareMathOperator{\hen}{HEN}

\DeclareMathOperator{\rad}{rad}
\DeclareMathOperator{\vcr}{vcr}
\DeclareMathOperator{\vrv}{vrv}
\DeclareMathOperator{\RVH}{RVH}

\DeclareMathOperator{\can}{\mathbf{c}}
\DeclareMathOperator{\wgt}{wgt}
\DeclareMathOperator{\pvf}{pvf}
\DeclareMathOperator{\prv}{prv}

\DeclareMathOperator{\der}{d}

\begin{document}

\title[Special transformations in ACVF]{Special transformations in algebraically closed valued fields}

\author{Yimu Yin}

\address{Department of Mathematics, University of Pittsburgh, 301 Thackeray Hall, Pittsburgh, PA  15260}

\email{yimuyin@pitt.edu}

\begin{abstract}
We present two of the three major steps in the construction of motivic integration,
that is, a homomorphism between Grothendieck semigroups that are
associated with a first-order theory of algebraically closed
valued fields, in the fundamental work of Hrushovski and
Kazhdan~\cite{hrushovski:kazhdan:integration:vf}. We limit our
attention to a simple major subclass of \vmin-minimal theories of
the form $\ACVF_S(0, 0)$, that is, the theory of algebraically closed
valued fields of pure characteristic $0$ expanded by a $(\VF,
\Gamma)$-generated substructure $S$ in the language $\lan{RV}$.
The main advantage of this subclass is the presence of syntax. It
enables us to simplify the arguments with many different technical
details while following the major steps of the Hrushovski-Kazhdan
theory.
\end{abstract}

\maketitle

\tableofcontents

\section{Introduction}

The theory of motivic integration in valued fields has been
progressing rapidly since its first introduction by Kontsevich.
Early developments by Denef and Loeser et al.\ have yielded many
important results in many directions. The reader is referred to
\cite{hales:2005} for an excellent introduction to the
construction of motivic measure.

There have been different approaches to motivic integration. The comprehensive study in Cluckers-Loeser~\cite{cluckers:loeser:constructible:motivic:functions}
has successfully united some major ones on a general foundation.
Their construction may be applied in general to the field of
formal Laurent series over a field of characteristic $0$ but
heavily relies on the Cell Decomposition Theorem of Denef-Pas~\cite{Denef:1986,Pa89}. We note that cell decomposition is also achieved in other cases, for example, in certain finite extensions of \p-adic fields~\cite{Pas:1990} and in henselian fields with respect to a first-order language that is equipped with, instead of an angular component, a collection of residue multiplicative structures~\cite{cluckers:loeser:bminimality}. On the other hand, the Hrushovski-Kazhdan integration theory~\cite{hrushovski:kazhdan:integration:vf} is a
major development that does not require the presence of an angular
component map and hence is of great foundational importance. Its
basic objects of study are models of the so-called \vmin-minimal theories, for example, the theory of algebraically closed valued fields of
pure characteristic $0$ and the theories of its rigid analytic expansions~\cite{lipsh:1993, lipsh:rob:1998}. The method of the Hrushovski-Kazhdan integration theory is based on a fine analysis of definable subsets up to definable bijections in a first-order language $\lan{RV}$ for valued fields. Of course the method of the Cluckers-Loeser approach~\cite{cluckers:loeser:constructible:motivic:functions} is similar, but the ``up to definable bijections'' point of view is not so much stressed. In fact both approaches are rooted in the Cohen-Denef analysis of definable sets that leads to cell decomposition~\cite{Co69, Denef:1986}.

The language $\lan{RV}$ has two sorts: the $\VF$-sort and the $\RV$-sort.
One of the main features of $\lan{RV}$ is that the residue field
and the value group are wrapped together in one sort $\RV$. Let $(K, \vv)$ be a
valued field and $\OO$, $\MM$, $\K$ the corresponding valuation
ring, its maximal ideal, and the residue field. Let $\RV(K) = K^{\times} / (1 + \MM)$ and $\rv : K^{\times} \fun \RV(K)$ the quotient map. Note that, for each $a \in K$, $\vv$ is constant on
the subset $a + a\MM$ and hence there is a naturally induced map
$\vrv$ from $\RV(K)$ onto the value group $\Gamma$. The situation
is illustrated in the following commutative diagram
\begin{equation*}
\bfig
 \square(0,0)/^{ (}->`->>`->>`^{ (}->/<600, 400>[\OO \mi \MM`K^{\times}`\K^{\times}`
\RV(K);`\text{quotient}`\rv`]
 \morphism(600,0)/->>/<600,0>[\RV(K)`\Gamma;\vrv]
 \morphism(600,400)/->>/<600,-400>[K^{\times}`\Gamma;\vv]
\efig
\end{equation*}
where the bottom sequence is exact.

Let $\VF_*[\cdot]$ and $\RV[*, \cdot]$ be two categories of definable sets that are respectively associated with the $\VF$-sort and the $\RV$-sort. In $\VF_*[\cdot]$, the objects are definable subsets of products of the form $\VF^n \times \RV^m$ and the morphisms are definable functions. On the other hand, for technical reasons (particularly for keeping track of dimensions), $\RV[*, \cdot]$ is formulated in a somewhat complicated way (see Section~\ref{section:category}). The main construction of the Hrushovski-Kazhdan theory is a canonical
homomorphism from the Grothendieck semigroup $\gsk \VF_*[\cdot]$
to the Grothendieck semigroup $\gsk \RV[*, \cdot]$ modulo a
semigroup congruence relation $\isp$ on the latter. In fact, it turns out to be an isomorphism. This construction has three main steps.
\begin{itemize}
 \item\emph{Step~1.} First we define a lifting map $\bb L$ from the set of the objects in $\RV[*, \cdot]$ into the set of the objects in $\VF_*[\cdot]$; see Definition~\ref{def:L}. Next we single out a subclass of
the isomorphisms in $\VF_*[\cdot]$, which are called special bijections; see
Definition~\ref{defn:special:bijection}. Then we show that for
any object $A$ in $\VF_*[\cdot]$ there is a special bijection
$T$ on $A$ and an object $\mathbf{U}$ in $\RV[*, \cdot]$ such that
$T(A)$ is isomorphic to $\bb L(\mathbf{U})$. This implies that $\bb L$
hits every isomorphism class of $\VF_*[\cdot]$. Of course, for
this result alone we do not have to limit our means to special
bijections. However, in Step~3 below, special bijections
become an essential ingredient in computing the congruence
relation $\isp$.

 \item\emph{Step~2.} For any two isomorphic objects $\mathbf{U}_1, \mathbf{U}_2$ in $\RV[*, \cdot]$, their lifts $\bb L(\mathbf{U}_1), \bb L(\mathbf{U}_2)$ in
$\VF_*[\cdot]$ are isomorphic as well. This shows that $\bb L$
induces a semigroup homomorphism from $\gsk \RV[*, \cdot]$
into $\gsk \VF_*[\cdot]$, which is also denoted by $\bb L$.

 \item\emph{Step~3.} A number of classical properties of integration can already be (perhaps only partially) verified for the inversion of the homomorphism $\bb L$ and hence, morally, this third step is not necessary. To facilitate computation in future applications, however, it seems much more satisfying to have a precise description of the semigroup congruence relation induced by it. The basic notion used in the description is that of a blowup of an object in $\RV[*, \cdot]$. We then show that, for any objects $\mathbf{U}_1, \mathbf{U}_2$ in $\RV[*, \cdot]$, there are isomorphic iterated blowups $\mathbf{U}_1^{\sharp}, \mathbf{U}_2^{\sharp}$ of $\mathbf{U}_1, \mathbf{U}_2$ if and only if $\bb L(\mathbf{U}_1), \bb L(\mathbf{U}_2)$ are isomorphic. The ``if'' direction essentially contains a form of Fubini's Theorem and is the most technically involved part of the construction.
\end{itemize}
The inverse of $\bb L$ thus obtained is a Grothencieck homomorphism. If the Jacobian transformation preserves integrals, that is, the change of variables formula holds, then it may be called a motivic integration. When the Grothendieck semigroups are formally groupified this integration is recast as a ring homomorphism.

In this paper we give a presentation of the first two steps. The sections are organized as follows. Throughout we shall follow the terminology and notation of~\cite{Yin:QE:ACVF:min}. For the reader's convenience some key definitions and notational conventions are recalled in Section~\ref{section:prelim}, where new ones are introduced as well. To delineate the basic geography of definable subsets, many structural properties concerning the three sorts $\VF$, $\RV$, and $\Gamma$ are needed. These are discussed in Section~\ref{section:some:stru} and Section~\ref{section:more:stru}. In Section~\ref{section:category} we first
discuss various notions of dimension, mainly $\VF$-dimension and
$\RV$-dimension, and then describe the relevant categories of definable
subsets and the formulation of their Grothendieck semigroups. The fundamental lifting map $\bb L$ between $\VF$-categories and $\RV$-categories is also introduced here. The central topic of Section~\ref{section:RV:product} is
$\RV$-pullbacks and special bijections on them. Corollary~\ref{all:subsets:rvproduct} corresponds to
Step~1 above. In Section~\ref{section:inter} we describe the ``descent'' technique and use it to obtain a general quantifier elimination result for henselian fields.

Section~\ref{section:lifting} is devoted to showing Step~2 above.
The notion of a $\lbar \gamma$-polynomial is introduced here, which generalizes the relation between a polynomial with coefficients in the valuation ring and its projection into the residue field. This leads to Lemma~\ref{hensel:lemma}, a
generalized form of the multivariate version of Hensel's lemma.
Note that in order to apply Lemma~\ref{hensel:lemma} to a given
definable subset we need to find suitable polynomials with a simple
common residue root. This is investigated in Lemma~\ref{exists:gamma:polynomial}, which does not hold when the substructure in question contains an excessive amount of parameters in the $\RV$-sort. This is the reason why motivic
integration is constructed only when parameters are taken from a $(\VF, \Gamma)$-generated substructure.

For finer categories of definable subsets that can handle the Jacobian transformation, a notion of the Jacobian is needed. This is provided in Section~\ref{section:diff}. Then in Section~\ref{section:cat:vol} we define these finer categories and explain how to carry out Step~1 and Step~2 for them.

While we do follow the broad outline of~\cite{hrushovski:kazhdan:integration:vf}, there are significant technical differences. To begin with, our construction is
specialized for $\ACVF_S(0, 0)$, that is the theory of algebraically closed valued fields of pure characteristic $0$, formulated in the
language $\lan{RV}$ and expanded by a substructure $S$, where $S$ is generated by elements in the field sort and the (imaginary) value group sort. For this simple major subclass of \vmin-minimal theories we are able
to work with syntax. Very often, in order to grasp the geometrical
content of a definable subset $A$, it is a very fruitful exercise to
analyze the logical structure of a typical formula that defines
$A$, especially when quantifier elimination is available. Consequently, in the context of this paper, syntactical analysis
affords tremendous simplifications of many lemmas in
\cite{hrushovski:kazhdan:integration:vf}. It also gives rise to technical tools that are especially powerful for $\ACVF_S(0, 0)$, the most important of which is Theorem~\ref{special:bi:polynomial:constant}.

Step~3 of the construction of motivic integration will be presented in a sequel.

\section{Preliminaries}\label{section:prelim}

Throughout this paper we shall use the terminology and notation introduced  in~\cite{Yin:QE:ACVF:min}. For the reader's convenience, we recall a few key definitions here.

\begin{defn}
The language $\lan{RV}$ has the following sorts and symbols:
\begin{enumerate}
 \item a $\VF$-sort, which uses the language of rings
 $\lan{R} = \set{0, 1, +, -, \times}$;
 \item an $\RV$-sort, which uses
  \begin{enumerate}
    \item the group language $\set{1, \times}$,
    \item two constant symbols $0$ and $\infty$,
    \item a unary predicate $\K^{\times}$,
    \item a binary function $+ : \K^2 \fun \K$ and a
    unary function $-: \K \fun \K$, where $\K =
    \K^{\times} \cup \set{0}$,
    \item a binary relation $\leq$;
    \end{enumerate}
  \item a function symbol $\rv$ from the $\VF$-sort into the $\RV$-sort.
\end{enumerate}
\end{defn}
The two sorts without the zero elements are respectively denoted by $\VF^{\times}$ and $\RV$; $\RV \mi \set{\infty}$ is denoted by $\RV^{\times}$; and $\RV \cup \set{0}$ is denoted by $\RV_0$.

\begin{defn}\label{defn:acvf}
\emph{The theory $\ACVF$ of algebraically closed valued fields in $\lan{RV}$} states the following:
\begin{enumerate}
 \item $(\VF, 0, 1, + , -, \times)$ is an algebraically close field;

 \item $(\RV^{\times}, 1, \times)$ is a divisible abelian
 group, where multiplication $\times$ is augmented by $t
 \times 0 = 0$ for all $t \in \K$ and $t \times \infty =
 \infty$ for all $t \in \RV_0$;

 \item $(\K, 0, 1, +, -, \times)$ is an algebraically closed field;

 \item the relation $\leq$ is a preordering on $\RV$ with
 $\infty$ the top element and $\K^{\times}$ the equivalence
class of 1;

 \item the quotient $\RV / \K^{\times}$, denoted as
 $\Gamma \cup \set{\infty}$, is a divisible ordered abelian
group with a top element, where the ordering and the group
operation are induced by $\leq$ and $\times$, respectively,
and the quotient map $\RV \fun \Gamma \cup \set{\infty}$ is
denoted as $\vrv$;

 \item the function $\rv : \VF^{\times} \fun \RV^{\times}$
 is a surjective group homomorphism augmented by $\rv(0) =
\infty$ such that the composite function
\[
\vv = \vrv \circ \rv : \VF \fun \Gamma \cup \set{\infty}
\]
is a valuation with the valuation ring $\OO =
\rv^{-1}(\RV^{\geq 1})$ and its maximal ideal $\MM =
\rv^{-1}(\RV^{> 1})$, where
\[
\RV^{\geq 1} = \set{x \in \RV: 1 \leq x}, \quad
\RV^{> 1} = \set{x \in \RV: 1 < x}.
\]
\end{enumerate}
\end{defn}

Semantically we shall treat $\Gamma$ as an imaginary sort and write $\RV_{\Gamma}$ for $\RV \cup \Gamma$. However, syntactically any reference to $\Gamma$ may be eliminated in the usual way and we shall still work with $\lan{RV}$-formulas.

\begin{thm}[{\cite[Theorem~3.10]{Yin:QE:ACVF:min}}]
The theory $\ACVF$ admits quantifier elimination.
\end{thm}

Since a $\VF$-sort literal can be equivalently expressed as an $\RV$-sort literal, we may assume that an $\lan{RV}$-formula contains no $\VF$-sort literals at all. In particular, we may assume that every $\VF$-sort polynomial $F(\lbar X)$ in a formula $\phi$ occurs in the form $\rv(F(\lbar X))$. This understanding sometimes makes the discussion more streamlined. We say that $F(\lbar X)$ is an \emph{occurring polynomial} of $\phi$.

\begin{defn}\label{def:normal:form}
Let $\lbar X$ be $\VF$-sort variables and $\lbar Y$ be $\RV$-sort variables.

A \emph{$\K$-term} is an $\lan{RV}$-term of the form $\sum_{i = 1}^k (\rv(F_{i}(\lbar X)) \cdot r_{i} \cdot \lbar Y^{\lbar n_i})$ with $k > 1$, where $F_{i}(\lbar X)$ is a polynomial with coefficients in $\VF$ and $r_{i} \in
\RV$. An \emph{$\RV$-literal} is an $\lan{RV}$-formula of the form
\[
\rv(F(\lbar X)) \cdot \lbar Y^{\lbar m} \cdot T(\lbar X, \lbar Y)
\, \Box \, \rv(G(\lbar X)) \cdot r \cdot \lbar Y^{\lbar l} \cdot S(\lbar X, \lbar Y),
\]
where $F(\lbar X)$, $G(\lbar X)$ are polynomials with coefficients in $\VF$, $T(\lbar X, \lbar Y)$, $S(\lbar X, \lbar Y)$ are $\K$-terms, $r \in \RV$, and $\Box$ is one of the symbols $=$, $\neq$,
$\leq$, and $>$.
\end{defn}

Note that if $T(\lbar X, \lbar Y)$ is a $\K$-term, $\lbar a \in \VF$, and $\lbar t \in \RV$ then $T(\lbar a, \lbar t)$ is defined if and only if each summand in $T(\lbar a, \lbar t)$ is either of value $1$ or is equal to $0$. Also, since the value of $\K$-terms are $0$, we may assume that they do not occur in $\RV$-sort inequalities.

Any $\lan{RV}$-formula with parameters is provably equivalent to a disjunction of conjunctions of $\RV$-literals. This follows from QE of $\ACVF$ and routine syntactical inductions.

Let $\ACVF(0, 0)$ denote $\ACVF$ with pure characteristic $0$. From now on we shall work in a sufficiently saturated model $\gC$ of $\ACVF(0, 0)$. Let $S \sub \gC$ be a small substructure such that \emph{$\Gamma(S)$ is nontrivial}. Let $\ACVF_S(0, 0)$ be the theory that extends $\ACVF(0, 0)$ with the atomic diagram of $S$. For notational simplicity we shall still refer to the language of $\ACVF_S(0, 0)$ as $\lan{RV}$. Although we do not include the multiplicative inverse function in the $\VF$-sort and the $\RV$-sort, we always assume that, without loss of generality, $\VF(S)$ is a field and $\RV^{\times}(S)$ is a group.

\begin{conv}
By a definable subset of $\gC$ we mean a $\0$-definable
subset in the theory $\ACVF_S(0, 0)$. If additional parameters are used in defining a subset then we shall spell them out explicitly if necessary.
\end{conv}

The substructure generated by a subset $A$ is denoted by $\la A \ra$ or $\dcl(A)$. The model-theoretic algebraic closure of $A$ is denoted by $\acl(A)$. A substructure $S$ is \emph{$\VF$-generated} if there is a subset $A \sub \VF(S)$ such that $S  = \la A \ra$; similarly for $(\VF, \RV)$-generated substructures, $(\VF, \Gamma)$-generated substructures, etc.

\begin{defn}
A subset $\gb$ of $\VF$ is an \emph{open ball} if there is a
$\gamma \in \Gamma$ and a $b \in \gb$ such that $a \in \gb$ if and
only if $\vv(a - b) > \gamma$. It is a \emph{closed ball} if $a
\in \gb$ if and only if $\vv(a - b) \geq \gamma$. It is an
\emph{$\rv$-ball} if $\gb = \rv^{-1}(t)$ for some $t \in \RV$. The
value $\gamma$ is the \emph{radius} of $\gb$, which is denoted as
$\rad (\gb)$. Each point in $\VF$ is a closed ball of radius $\infty$ and $\VF$ is a clopen ball of radius $- \infty$.

If $\vv$ is constant on $\gb$ --- that is, $\gb$ is contained in an $\rv$-ball --- then $\vv(\gb)$ is the \emph{valuative center} of $\gb$; if $\vv$ is not constant on
$\gb$, that is, $0 \in \gb$, then the \emph{valuative center} of
$\gb$ is $\infty$. The valuative center of $\gb$ is denoted by
$\vcr(\gb)$.

A subset $\gp \sub \VF^n \times \RV^m$ is an (\emph{open, closed, $\rv$-}) \emph{polydisc} if it is of the form $(\prod_{i \leq n} \gb_i) \times \set{\lbar t}$, where each $\gb_i$ is an (open, closed, $\rv$-) ball and $\lbar t \in \RV^m$.
If $\gp$ is a polydisc then the \emph{radius} of $\gp$, denoted as $\rad(\gp)$,
is $\min \set{\rad(\gb_i) : i \leq n}$. The open and closed polydiscs centered at a sequence of elements $\lbar a = (a_1, \ldots, a_n) \in \VF^n$ with radii $\lbar \gamma = (\gamma_1, \ldots, \gamma_n) \in \Gamma^n$ are respectively denoted as $\go(\lbar a, \lbar \gamma)$ and $\gc(\lbar a, \lbar \gamma)$.

An $\rv$-polydisc $\rv^{-1}(t_1, \ldots, t_n) \times \{ \lbar s \}$ is \emph{degenerate} if $t_i = \infty$ for some $i$.
\end{defn}

\begin{defn}
Let $\LL$ be a language expanding $\lan{RV}$. Let $M$ be a structure of $\LL$ that satisfies the axioms for valued fields. We say that $M$ is \emph{\cmin-minimal} if every parametrically definable subset of $\VF(M)$ is a boolean combination of balls. An $\LL$-theory $T$ is \emph{\cmin-minimal} if every model of $T$ is \emph{\cmin-minimal}.
\end{defn}

\begin{thm}[{\cite[Theorem~4.2]{Yin:QE:ACVF:min}}]\label{c:min:acvf}
The theory $\ACVF$ is $C$-minimal.
\end{thm}

\begin{nota}
We sometimes write $\lbar a \in \VF$ to mean that every element in
the tuple $\lbar a$ is in $\VF$; similarly for $\RV$, $\Gamma$, etc. We often write
$(\lbar a, \lbar t)$ for a tuple of elements with the understanding that $\lbar a \in \VF$ and $\lbar t \in \RV$. For
such a tuple $(\lbar a, \lbar t) = (a_1, \ldots, a_n, t_1, \ldots
t_m)$, let
\[
\rv(\lbar a, \lbar t) = (\rv(a_1), \ldots, \rv(a_n), \lbar t), \quad
\rv^{-1}(\lbar a, \lbar t) = \set{\lbar a} \times \rv^{-1}(t_1)
\times \cdots \times \rv^{-1}(t_m),
\]
similarly for other functions.

Let $\lbar a = (a_1, \ldots, a_n)$, $\lbar a' = (a'_1, \ldots, a'_n)$ be tuples in $\VF$. We write $\vv(\lbar a - \lbar a')$ for the element
\[
\min \set{\vv(a_i - a_i') : 1 \leq i \leq n} \in \Gamma.
\]
For any $\lbar \gamma =  (\gamma_1, \ldots, \gamma_n) \in \Gamma$, the open polydisc $\set{(b_1, \ldots, b_n) : \vv(b_i - a_i) > \gamma_i}$ is denoted by $\go(\lbar a, \lbar \gamma)$ and the closed polydisc $\set{(b_1, \ldots, b_n) : \vv(b_i - a_i) \geq \gamma_i}$ is denoted by $\gc(\lbar a, \lbar \gamma)$. We set $\go(\lbar a, \infty) = \gc(\lbar a, \infty) = \set{\lbar a}$.
\end{nota}

\begin{nota}\label{indexing}
Coordinate projection maps are ubiquitous in this paper. To
facilitate the discussion, certain notational conventions about
them are adopted.

Let $A \sub \VF^n \times \RV^m$. For any $n \in \N$, let $I_n
= \set{1, \ldots, n}$. First of all, the $\VF$-coordinates and the
$\RV$-coordinates of $A$ are indexed separately. It is cumbersome
to actually distinguish them notationally, so we just assume that
the set of the $\VF$-indices is $I_n$ and the set of the $\RV$-indices is $I_m$. This should never cause confusion in context.

Let $I = I_n \uplus I_m$, $E \sub I$, and $\tilde{E} = I \mi E$. If $E$ is a
singleton $\set{i}$ then we always write $E$ as $i$ and
$\tilde{E}$ as $\tilde{i}$. We write $\pr_E(A)$ for the projection
of $A$ to the coordinates in $E$. For any $\lbar a \in
\pr_{\tilde{E}} (A)$, the fiber $\{\lbar b : (\lbar b, \lbar a)
\in A \}$ is denoted by $\fib(A, \lbar a)$. Note that we shall often tacitly identify the two subsets $\fib(A, \lbar a)$ and $\fib(A, \lbar a) \times \set{\lbar
a}$. Also, it is often more convenient to use simple descriptions
as subscripts. For example, if $E = \set{1, \ldots, k}$ etc.\ then
we may write $\pr_{\leq k}$ etc. If $E$ contains exactly the
$\VF$-indices (respectively $\RV$-indices) then $\pr_E$ is written
as $\pvf$ (respectively $\prv$). If $E'$ is a subset of the coordinates of $\pr_E (A)$ then the composition $\pr_{E'} \circ \pr_E$ is written as $\pr_{E, E'}$. Naturally
$\pr_{E'} \circ \pvf$ and $\pr_{E'} \circ \prv$ are written as $\pvf_{E'}$ and $\prv_{E'}$, respectively.
\end{nota}

\section{Some structural properties}\label{section:some:stru}

In this section we shall list a number of structural properties concerning the relation among the three sorts $\VF$, $\RV$, and $\Gamma$. Some simple ones are just consequences of variations of compactness, for example:

\begin{lem}
Let $A$ be a definable subset and $s$ an element such that $s \in \acl(a)$ for every $a \in A$, then $s \in \acl(\0)$.
\end{lem}
\begin{proof}
By compactness, there are a definable partition $A_1, \ldots, A_m$ of $A$, integers $k_1, \ldots, k_m$, formulas $\phi_1(X, Y), \ldots, \phi_m(X, Y)$, such that if $a \in A_i$ then the subset $U_{a}$ defined by the formula $\phi_i(a, Y)$ contains $s$ and its size is at most $k_i$. Then $\bigcap_{a \in A} U_{a}$ is a definable finite subset that contains $s$.
\end{proof}

\begin{cor}\label{acl:VF:transfer:acl:RV}
For any $\lbar t \in \RV$, any $\lbar t$-definable subset $A \sub \rv^{-1}(\lbar t)$, and any element $x$, if $x \in \acl(\lbar a)$ for every $\lbar a \in A$ then $x \in \acl(\lbar t)$. Similarly, for any $\lbar \gamma \in \Gamma$, any $\lbar \gamma$-definable subset $B \sub \vrv^{-1}(\lbar \gamma)$, and any element $x$, if $x \in \acl(\lbar t)$ for every $\lbar t \in B$ then $x \in \acl(\lbar \gamma)$.
\end{cor}

For any $A \sub \VF$ let $A^{\alg}$ be the field-theoretic algebraic closure of $A$. The field generated by $\lbar a \in \VF$ is written as $\VF(S)(\lbar a)$.

\begin{lem}\label{dcl:to:ac}
For any $\lbar a$, $b \in \VF$ and $\lbar t \in \RV$, if $b \in \acl(\lbar a, \lbar t)$ then $b \in \VF(S)(\lbar a)^{\alg}$.
\end{lem}
\begin{proof}
Suppose for contradiction $b \notin \VF(S)(\lbar a)^{\alg}$. Let $\phi(X, \lbar a, \lbar t)$ be a formula that defines a finite subset containing $b$. Then, for any occurring polynomial $F(X, \lbar a)$ of $\phi(X, \lbar a, \lbar t)$, we have $F(b, \lbar a) \neq 0$. We see that, for any $d \in \VF$, if $\vv(d - b)$ is sufficiently large then $\rv(F(d, \lbar a)) = \rv(F(b, \lbar a))$ for all occurring polynomials $F(X, \lbar a)$ and hence $\phi(d, \lbar a, \lbar t)$ holds, which is a contradiction.
\end{proof}

\begin{cor}
For any $\lbar a \in \VF$ and $B \sub \RV$, the transcendental degrees of $\VF(S)(\lbar a)$, $\VF(\la \lbar a, B \ra)$, and $\VF(\acl(\lbar a, B ))$ over $\VF(S)$ are all equal.
\end{cor}

\begin{cor}[{\cite[Lemma~4.12]{Yin:QE:ACVF:min}}]\label{function:rv:to:vf:finite:image}
Let $A \sub \RV^m$ and $f: A \fun \VF^n$ a definable function. Then $f(A)$ is finite.
\end{cor}
\begin{proof}
We may assume $n=1$. Since for any $\lbar t \in A$ we have $f(\lbar t) \in \la \lbar t \ra$, by Lemma~\ref{dcl:to:ac}, $f(\lbar t) \in \VF(S)^{\alg}$. By compactness $f(A)$ must be finite.
\end{proof}

\begin{lem}[{\cite[Lemma~4.3]{Yin:QE:ACVF:min}}]\label{exchange}
The exchange principle holds in both sorts:
\begin{enumerate}
 \item For any $a$, $b \in \VF$, if $a \in \acl(b) \mi \acl(\0)$ then $b \in \acl(a)$.
 \item For any $t$, $s \in \RV$, if $t \in \acl(s) \mi \acl(\0)$ then $s \in \acl(t)$.
\end{enumerate}
\end{lem}

\begin{cor}\label{alg:ind:imag}
If $a \in \VF$ is such that $a \notin \acl(\0)$, then for any $t \in \RV$ we have $a \notin \acl(t)$. Similarly, if $t \in \RV$ is such that $t \notin \acl(\0)$, then for any $\gamma \in \Gamma$ we have $t \notin \acl(\gamma)$.
\end{cor}
\begin{proof}
For the first claim, suppose for contradiction that $a \in \acl(t)$. Then $a \in \acl(b)$ for every $b \in \rv^{-1}(t)$. So by the exchange principle we have $b \in \acl(a)$ for every $b \in \rv^{-1}(t)$, which is impossible. The other claim is proved in the same way.
\end{proof}

\begin{lem}[{\cite[Lemma~4.9]{Yin:QE:ACVF:min}}]\label{average:0:rv:not:constant}
Let $c_1, \ldots, c_k \in \VF$ be distinct elements of the same
value $\alpha$ such that their average is $0$. Then for some $c_i
\neq c_j$ we have $\vv(c_i -c_j) = \alpha$ and hence $\rv$ is
not constant on the set $\set{c_1, \ldots, c_k}$.
\end{lem}

\begin{lem}[{\cite[Lemma~4.10]{Yin:QE:ACVF:min}}]\label{finite:VF:project:RV}
Let $A$ be a definable finite subset of $\VF^n$. Then there is a definable injection $f : A \fun \RV^m$ for some $m$.
\end{lem}

\begin{lem}[{\cite[Lemma~4.15]{Yin:QE:ACVF:min}}]\label{effectiveness}
Let $\gB$ be an algebraic set of closed balls. Then $\gB$ has centers.
\end{lem}

\begin{lem}\label{ball:proper:sub:center}
If a ball contains a definable proper subset then it contains a definable point.
\end{lem}
\begin{proof}
The proof of~\cite[Lemma~4.16]{Yin:QE:ACVF:min} works almost verbatim here.
\end{proof}

\begin{cor}\label{ball:RV:finite}
Let $B \sub \RV$ and $f : \rv^{-1}(B) \fun \RV^m$ a definable function. Then, for all but finitely many $t \in B$, $f \rest \rv^{-1}(t)$ is constant.
\end{cor}
\begin{proof}
For any  $t \in B$, if $f \rest \rv^{-1}(t)$ is not constant then, by Lemma~\ref{ball:proper:sub:center}, for each $\lbar s \in \ran(f \rest \rv^{-1}(t))$, $\rv^{-1}(t)$ contains a $(\lbar t, \lbar s)$-definable point $a_{\lbar t, \lbar s}$. By Corollary~\ref{function:rv:to:vf:finite:image}, the image of the function given by $(\lbar t, \lbar s) \efun a_{\lbar t, \lbar s}$ is finite.
\end{proof}

\begin{lem}[{\cite[Lemma~4.17]{Yin:QE:ACVF:min}}]\label{algebraic:balls:definable:centers}
Suppose that $S$ is $(\VF, \Gamma)$-generated. Let $\gB$ be an algebraic set of balls. Then $\gB$ has centers.
\end{lem}

\begin{cor}[{\cite[Corollary~4.18]{Yin:QE:ACVF:min}}]\label{aclS:model}
Suppose that $S$ is $\VF$-generated. If the value group $\Gamma(\acl(S))$ is nontrivial then $\acl(S)$ is a model of $\ACVF_S(0,0)$.
\end{cor}

\begin{lem}\label{inf:RV:def:gam}
Let $A$ be a definable subset of $\RV$. Let $V \sub \Gamma$ be the subset such that $\gamma \in V$ if and only if $\vrv^{-1}(\gamma) \cap A$ is nonempty and finite. Then $V$ is finite and definable.
\end{lem}
\begin{proof}
By \cmin-minimality each $\vrv^{-1}(\gamma) \cap A$ is either finite or cofinite. By compactness there is a number $k$ such that if $\vrv^{-1}(\gamma) \cap A$ is finite then it has at most $k$ elements. So $V$ is definable. By \cmin-minimality again $V$ must be finite.
\end{proof}

Let $A$ be a subset and $B \sub A \times \VF^n \times \RV^m$. We say that $B$ is a \emph{subset over $A$} if the projection of $B$ to $A$ is surjective.

\begin{nota}\label{fun:into:power}
Let $A_1$, $A_2$ be subsets and $R_1$, $R_2$ equivalence relations on them, respectively. A subset $B \sub A_1 \times A_2$ over $A_1$ may be considered as a function from $A_1 / R_1$ into the powerset $\mdl P(A_2 / R_2)$ if, for each equivalence class $C \in A_1 / R_1$ and every $c_1$, $c_2 \in C$, there is a $U \in \mdl P(A_2 / R_2)$ such that $\fib(B, c_1) = \fib(B, c_2) = \bigcup U$. In this case, we sometime do write $B$ as a function $A_1 / R_1 \fun \mdl P(A_2 / R_2)$. We are of course only interested in definable objects. For example, we will discuss functions of the forms
\[
\VF / \MM \fun \mdl P(\RV^m), \,\, \VF^n \times \Gamma^l \fun \mdl P(\RV^m).
\]
\end{nota}

More elaborate syntactical analysis using the normal forms in Definition~\ref{def:normal:form} can sometimes reveal finer details.

\begin{lem}\label{fun:bounded:cons}
Let $f : \VF^{\times} \fun \mdl P(\RV^m)$ be a definable function such that the subset $\vrv(\bigcup f(\VF^{\times}))$ is bounded from both above and below. Then for any sufficiently large $\delta \in \Gamma$ the restriction $f \rest \go(0, \delta) \mi \set{0}$ is constant.
\end{lem}
\begin{proof}
Let $\phi(X, \lbar Y)$ be a disjunction of conjunctions of $\RV$-literals that defines $f$. For any $\delta \in \Gamma$ let $\phi_{\delta}(X, \lbar Y)$ be the formula $\phi(X, \lbar Y) \wedge \vv(X) > \delta$. Any term of the form $\rv(F(X))$ in $\phi(X, \lbar Y)$ may be written as $\rv(X^m F^*(X))$, where $F^*(0) \neq 0$. So if $\vv(a)$ is sufficiently large then
\[
\rv(a^m F^*(a)) = \rv(a^m) \rv(F^*(0)).
\]
Since $\vrv(\bigcup f(\VF^{\times}))$ is bounded from below, if $\delta$ is sufficiently large then we may assume that no $\K$-term in $\phi_{\delta}(X, \lbar Y)$ contains $X$. Since $\vrv(\bigcup f(\VF^{\times}))$ is also bounded from above, it is not hard to see that $\phi_{\delta}(X, \lbar Y)$ is actually equivalent to a formula of the form $\psi(\lbar Y) \wedge \vv(X) > \delta$, where $\psi(\lbar Y)$ does not contain $X$.
\end{proof}

It is not hard to see that the same argument shows that the above lemma also holds for functions $f : \VF^{\times} \fun \mdl P(\RV^m)$ that satisfy the obvious condition.

\begin{lem}\label{fun:quo:rep}
Let $G$ be a definable additive subgroup of $\VF$ (hence $G$ is either an open ball around $0$ or a closed ball around $0$). Let $f : \VF \fun \mdl P(\RV^m)$ be a definable function. Then
\begin{enumerate}
  \item There are $G$-cosets $D_1, \ldots, D_n$ such that $f \rest (\VF \mi \bigcup_i D_i)$ is a function from $(\VF \mi \bigcup_i D_i) / G$ into $\mdl P(\RV^m)$.
  \item If either $G$ is a closed ball or $S$ is $(\VF, \Gamma)$-generated then there is a definable function $f_{\downarrow} : \VF / G \fun \mdl P(\RV^m)$ such that for any $G$-coset $D$ there is a $d \in D$ such that $f(d) = f_{\downarrow}(D)$.
\end{enumerate}
\end{lem}
\begin{proof}
For any $D \in \VF / G$ and any $\lbar t \in \RV^m$ let $U_{\lbar t}(D) = \set{d \in D : \lbar t \in f(d)}$. Let
\[
E_{\lbar t} = \set{D \in \VF / G : U_{\lbar t}(D) \neq \0 \text{ and } U_{\lbar t}(D) \neq D}.
\]
Note that $E_{\lbar t}$ is $\lbar t$-definable. Let $A = \{\lbar t \in \RV^m : E_{\lbar t} \neq \0 \}$, which is definable. If $D \notin E_{\lbar t}$ for any $\lbar t$ then $f \rest D$ is constant. So, without loss of generality, $A \neq \0$. For any $\lbar t \in A$, by $C$-minimality and compactness, there is a $\lbar t$-definable function $h_{\lbar t}$ on $E_{\lbar t}$ such that, for each $D \in E_{\lbar t}$,
\begin{enumerate}
 \item $h_{\lbar t}(D)$ is either the union of the positive boolean components of $U_{\lbar t}(D)$ or the union of the negative boolean components of $U_{\lbar t}(D)$,
 \item there is a $D$-definable closed ball $\gb_{D} \sub D$ that properly contains $h_{\lbar t}(D)$.
\end{enumerate}
Since $h_{\lbar t}(E_{\lbar t})$ is $\lbar t$-definable, by $C$-minimality again, $E_{\lbar t}$ must be finite. By Lemma~\ref{effectiveness}, there is a $\lbar t$-definable subset $A_{\lbar t}$ such that $\abs{A_{\lbar t} \cap \gb_D} = 1$. Let $g_D : A \fun \VF$ be the $D$-definable function given by $\lbar t \efun A_{\lbar t} \cap \gb_D$ if $D \in E_{\lbar t}$ and $\lbar t \efun 0$ otherwise. By Corollary~\ref{function:rv:to:vf:finite:image}, $g_D(A)$ is finite. Since $g_D(A) \sub D \cup \set{0}$, by $C$-minimality, the definable subset $\bigcup_{D \in \VF / G} g_D(A)$ must be finite and hence $\bigcup_{\lbar t \in A} E_{\lbar t}$ is finite. This establishes~(1). By Lemma~\ref{effectiveness} or Lemma~\ref{algebraic:balls:definable:centers}, $\bigcup_{\lbar t \in A} E_{\lbar t}$ has definable centers. This establishes~(2).
\end{proof}

\begin{rem}\label{fun:quo:rep:mult}
Let $G$ be a definable multiplicative subgroup of $\VF^{\times}$. Then $G$ is an open ball around $1$ or a closed ball around~$1$ or $\OO \mi \MM$. It is easy to see that if $G$ is not $\OO \mi \MM$ then the proof of Lemma~\ref{fun:quo:rep} also works with respect to $G$. If $G$ is $\OO \mi \MM$ then we can modify the proof as follows: in the construction of $h_{\lbar t}$, $\gb_{D} \sub D$ is a finite union of $\rv$-balls and contains $h_{\lbar t}(D)$.
\end{rem}

\section{Categories of definable subsets}\label{section:category}

\subsection{Dimensions}

For the categories of definable sets associated with $\ACVF_S(0,0)$ and their Grothendieck groups, two notions of dimension with respect to the two sorts are needed. Some basic properties of them are stated below.

Let $A \sub \VF^n \times \RV^m$ be a definable subset.

\begin{defn}
The \emph{$\VF$-dimension} of $A$, denoted by $\dim_{\VF}(A)$, is the
smallest number $k$ such that there is a definable finite-to-one function $f: A \fun \VF^k \times \RV_{\Gamma}^l$.
\end{defn}

\begin{lem}\label{dim:RV:fibers}
For any natural number $k$, $\dim_{\VF}(A) \leq k$ if and only if there is a definable injection $f : A \fun \VF^k \times \RV_{\Gamma}^l$ for some $l$.
\end{lem}
\begin{proof}
Suppose that $\dim_{\VF}(A )\leq k$. Let $g : A \fun \VF^k \times
\RV_{\Gamma}^l$ be a definable finite-to-one function. For every $(\lbar a,
\lbar t) \in g(A)$, since $g^{-1}(\lbar a, \lbar t)$ is finite, by
Lemma~\ref{finite:VF:project:RV}, there is an $(\lbar a, \lbar t)$-definable injection $h_{\lbar a, \lbar t} : g^{-1}(\lbar a,
\lbar t) \fun \RV_{\Gamma}^j$ for some $j$. By compactness, there is a
definable function $h : A \fun \RV_{\Gamma}^j$ for some $j$ such that $h
\rest g^{-1}(\lbar a, \lbar t)$ is injective for every $(\lbar a,
\lbar t) \in g(A)$. Then the function $f$ on $A$ given by
\[
(\lbar b, \lbar s) \efun (g(\lbar b, \lbar s), h(\lbar b, \lbar s))
\]
is as desired. The other direction is trivial.
\end{proof}

\begin{lem}\label{dim:VF:pullback}
Let $f : A \fun \RV_{\Gamma}^l$ be a definable function. Then
\[
\dim_{\VF}(A) = \max \{\dim_{\VF} (f^{-1}(\lbar t)) : \lbar t \in \RV_{\Gamma}^l \}.
\]
\end{lem}
\begin{proof}
Let $\max \{ \dim_{\VF} (f^{-1}(\lbar t)) : \lbar t \in \RV_{\Gamma}^l \} =
k$. By Lemma~\ref{dim:RV:fibers}, for every $\lbar t \in \ran(f)$,
there is a $\lbar t$-definable injective function $h_{\lbar t} :
f^{-1}(\lbar t) \fun \VF^k \times \RV_{\Gamma}^j$ for some $j$. By
compactness, there is a definable function $h : A \fun \VF^k
\times \RV_{\Gamma}^j$ for some $j$ such that $h \rest f^{-1}(\lbar t)$ is
injective for every $\lbar t \in \ran(f)$. Then the function on $A$
given by $(\lbar b, \lbar s) \efun (h(\lbar b, \lbar s), f(\lbar
b, \lbar s))$ is injective and hence $\dim_{\VF}(A) \leq k$. The
other direction is trivial.
\end{proof}

For any $(\lbar a, \lbar t) \in A$ let $\td(\lbar a, \lbar t)$ be the transcendental degree of $\VF(\seq{\lbar a})$ over $\VF(S)$. Let $\td(A) = \max \{\td(\lbar a, \lbar t) : (\lbar a, \lbar t) \in A \}$.

\begin{lem}\label{VF:dim:tran:deg}
$\dim_{\VF}(A) = \td(A)$.
\end{lem}
\begin{proof}
Let $\dim_{\VF}(A) = k$ and $\td(A) = k'$. By Lemma~\ref{dim:RV:fibers}, there is a definable injection $f : A \fun \VF^k \times \RV_{\Gamma}^l$ for some $l$. For any $(\lbar a, \lbar t) \in A$, if $f(\lbar a, \lbar t) = (\lbar b, \lbar s)$ then, by Lemma~\ref{dcl:to:ac}, $\VF(\seq{\lbar a}) \sub \VF(S)(\lbar b)^{\alg}$ and hence $\td(\lbar a, \lbar t) \leq k$. So $k' \leq k$.

On the other hand, for any $\lbar a = (a_1, \ldots, a_n) \in \pvf(A)$, there is a subset $E \sub \set{1, \ldots, n}$ of size $k'$ such that for any $j \in \tilde{E}$ we have $a_j \in \VF(\la \pr_E(\lbar a) \ra)^{\alg}$. Therefore, by compactness, there are a partition $A_i$ of $\pvf(A)$, subsets $E_i \sub \set{1, \ldots, n}$ of size $k'$, and formulas $\phi_i(\lbar X, \lbar Y)$ such that, for every $\lbar a \in A_i$, the subset $B_i \sub \VF^{n - k'}$ defined by $\phi_i(\lbar X, \pr_{E_i}(\lbar a))$ is finite and $\pr_{\tilde{E}_i}(\lbar a) \in B_i$. By compactness and Lemma~\ref{finite:VF:project:RV}, there is a definable injection $A \fun \VF^{k'} \times \RV_{\Gamma}^l$ for some $l$ and hence $k \leq k'$.
\end{proof}

It follows that additional parameters cannot change the $\VF$-dimension of a definable subset and hence there is no need to specify parameters when we discuss $\VF$-dimension.

\begin{cor}\label{VF:dim:mono}
If $f : A \fun \mdl P (\VF^{n'} \times \RV^{m'})$ is a definable function with finite images then $\dim_{\VF}(A) \geq \dim_{\VF}(\bigcup f(A))$.
\end{cor}

\begin{lem}\label{full:dim:open:poly}
$\dim_{\VF} (A) = n$ if and only if there is a $\lbar t \in \RV^m$ such that $\fib(A, \lbar t)$ contains an open polydisc.
\end{lem}
\begin{proof}
The ``if'' direction is immediate by Lemma~\ref{VF:dim:tran:deg}. For the ``only if'' direction, by compactness, it is enough to show the case $A \sub \VF^n$. We do induction on $n$. For the base case $n=1$, since $A$ is infinite, the lemma simply follows from $C$-minimality. We proceed to the inductive step $n = m + 1$. For each $\lbar a \in \pr_{\leq m} (A) = B$, let $\Delta_{\lbar a}$ be the subset of those $\gamma \in \Gamma$ such that $\fib(A, \lbar a)$ contains an open ball of radius $\gamma$ (if $\fib(A, \lbar a)$ is finite then we set $\Delta_{\lbar a} = \set{\infty}$). Since $\Gamma$ is $o$-minimal, some element $\gamma_{\lbar a}$ in $\Delta_{\lbar a}$ is $\lbar a$-definable. By compactness and the inductive hypothesis, we may assume that $\dim_{\VF} (B) = m$ and there is a quantifier-free formula $\phi(Z, \lbar X)$ such that, for every $\lbar a \in B$, $\fib(A, \lbar a)$ contains an open ball whose radius $\gamma_{\lbar a}$ is defined by the formula $\phi(Z, \lbar a)$.

Let $G_i(\lbar X)$ be the occurring polynomials of $\phi(Z, \lbar X)$. Let $f : B \fun \RV^k$ be the definable function given by
\[
\lbar a \efun (\rv(G_1(\lbar a)), \ldots, \rv(G_k(\lbar a))).
\]
By Lemma~\ref{dim:VF:pullback}, for some $\lbar t \in \RV^k$, $\dim_{\VF}(f^{-1}(\lbar t)) = m$. By the inductive hypothesis, $f^{-1}(\lbar t)$ contains an open polydisc $\gp$. Note that, by the construction of $f$, for every $\lbar a \in \gp$ the formula $\phi(Z, \lbar a)$ defines the same element $\delta \in \Gamma$. Let $\lbar b \in \gp$. We may assume that $\gp$ is $\lbar b$-definable. Note that, by Lemma~\ref{VF:dim:tran:deg}, the $\VF$-dimension of $\gp$ with respect to the substructure $\dcl(\lbar b)$ is still $m$. Consider the $\lbar b$-definable subset
\[
W = \set{(\lbar a, c) \in A : \lbar a \in \gp \text{ and } \go(c, \delta) \sub \fib(A, \lbar a)}.
\]
Since there is a $\lbar d \in W$ such that the transcendental degree of $\VF(\dcl(\lbar d, \lbar b))$ over $\VF(\dcl(\lbar b))$ is $m+1$, by Lemma~\ref{VF:dim:tran:deg} again, $\dim_{\VF}(W) = m + 1$. By compactness, for some $c \in \pr_{m+1} (W)$, $\dim_{\VF} (\fib(W, c)) = m$. By the inductive hypothesis (with respect to the substructure $\dcl(\lbar b, c)$), $\fib(W, c)$ contains an open polydisc $\gq$. So $\go(c, \delta) \times \gq \sub A$, as required.
\end{proof}

\begin{cor}\label{VF:dim:rv:polydisc}
Suppose that $A$ contains an $\rv$-polydisc of the form
\[
\set{(0, \ldots, 0)} \times \rv^{-1}(\lbar t) \times \set{\lbar s},
\]
where $\lbar t \in (\RV^{\times})^k$. Then $\dim_{\VF}(A) \geq k$.
\end{cor}

\begin{prop}\label{dim:vf:same:zar}
Suppose that $A \sub \VF^n$. Let $\overline A$ be the Zariski closure of $A$ and $k$ the Zariski dimension of $\overline A$. Then $\dim_{\VF}(A) = k$.
\end{prop}
\begin{proof}
Let $D$ be an irreducible component of $\overline A$ and $\lbar a \in D \cap A$. Let $P$ be the prime ideal of $\VF(S)^{\alg}[X_1, \ldots, X_n]$ such that $D = Z(P)$. Let $K_P$ be the corresponding quotient field. By general facts of commutative algebra (see, for example,~\cite[Chapter~11]{atiyah:mac}), the dimension of $D$ is equal to the transcendental degree of $K_P$ over $\VF(S)$. Since the latter is no less than the transcendental degree of $\VF(S)^{\alg}(\lbar a)$ over $\VF(S)$, we see that, by Lemma~\ref{VF:dim:tran:deg}, $k \geq \dim_{\VF}(A)$.

Let $\dim_{\VF}(A) =  \td(A) = k'$. If $k' = n$ then obviously $\overline A = \VF^n$ and hence $k = n$. Suppose $\dim_{\VF}(A) < n$. By compactness, there are Zariski closed subsets $D_i$ given by formulas of the form
\[
\bigwedge_{j \notin I_i} F_j(X_{i(1)}, \ldots, X_{i(k')}, X_j) = 0,
\]
where $I_i = \set{i(1), \ldots, i(k')}$ and each $F_j$ is a nonzero polynomial with coefficients in $\VF(S)$, such that $A \sub \bigcup_i D_i$. Then $\overline A \sub \bigcup_i D_i$ and hence each irreducible component of $\overline A$ is contained in some $D_i$, which implies $k \leq k'$.
\end{proof}

\begin{defn}
Let $B \sub \RV^m$ be a definable subset. The \emph{$\RV$-dimension} of $B$, denoted by $\dim_{\RV}(B)$, is the smallest number $k$ such that there is a definable finite-to-one function $f: B \fun \RV^k$ ($\RV^0$ is taken to be the singleton $\{\infty\}$).
\end{defn}

By the exchange principle (Lemma~\ref{exchange}), if $\dim_{\RV}(B) = k$ then for every $\lbar t \in B$ there is a subsequence $\lbar t' \sub \lbar t$ of length $k$ such that $\lbar t \in \acl(\lbar t')$. Also, by compactness, there is a $\lbar t \in B$ that contains an algebraically independent subsequence of length $k$ (in the model-theoretic sense); that is, for some subsequence $(t_{i(1)}, \ldots, t_{i(k)}) \sub \lbar t$ of length $k$, no $t_{i(j)}$ is in the algebraic closure of the other $k-1$ elements. So additional parameters cannot change the $\RV$-dimension of $B$ as well. Also, if $f : B \fun \mdl P (\RV^l)$ is a definable function then $\dim_{\RV}(B) \geq \dim_{\RV}(\bigcup f(B))$.

\begin{lem}\label{dim:rv:same:zar}
Let $\lbar s = (s_1, \ldots, s_m) \in \RV$, $\lbar \gamma = \vrv(\lbar s)$, and $B \sub \vrv^{-1}(\lbar \gamma)$ a definable subset. Let
\[
B_{\lbar s} = \set{(t_1/s_1, \ldots, t_m / s_m) : (t_1, \ldots, t_m) \in B}.
\]
Then $\dim_{\RV}(B)$ agrees with the Zariski dimension of $B_{\lbar s}$.
\end{lem}
\begin{proof}
The proof of Proposition~\ref{dim:vf:same:zar} works almost verbatim here.
\end{proof}

\begin{lem}\label{rv:dim:gamma:coset}
Let $B \sub \RV^m$ with $\dim_{\RV}(B) = k$. Then there is a definable sequence $\lbar \gamma \in \Gamma^m$ such that $\dim_{\RV}(B \cap \vrv^{-1}(\lbar \gamma)) = k$.
\end{lem}
\begin{proof}
By compactness, without loss of generality, we may assume that, for every $\lbar t \in B$, $\lbar t \in \acl(t_1, \ldots, t_k)$. Let
\[
B_0 = \set{(\pr_{\leq k} (\lbar t), \vv(\pr_{>k}(\lbar t))): \lbar t \in B} \sub \RV^k \times \Gamma^{m-k}.
\]
Clearly there is a natural number $q$ such that $\abs{\fib(B_0, \lbar t)} \leq q$ for every $\lbar t \in \pr_{\leq k} (B)$. For every $(\lbar t, \lbar \gamma) \in \pr_{>1}(B_0)$ let $D_{\lbar t, \lbar \gamma} \sub \Gamma$ be the subset such that $\alpha \in D_{\lbar t, \lbar \gamma}$ if and only if $\vrv^{-1}(\alpha) \cap \fib(B_0, (\lbar t, \lbar \gamma))$ is infinite. Since $\dim_{\RV}(B) = k$, by Corollary~\ref{alg:ind:imag}, we see that $D_{\lbar t, \lbar \gamma}$ is not empty for some $(\lbar t, \lbar \gamma) \in \pr_{>1}(B_0)$. Also, by Lemma~\ref{inf:RV:def:gam}, $D_{\lbar t, \lbar \gamma}$ is $(\lbar t, \lbar \gamma)$-definable. So, by compactness, the subset
\[
B_1 = \bigcup_{(\lbar t, \lbar \gamma) \in \pr_{>1}(B_0)} D_{\lbar t, \lbar \gamma} \times \set{(\lbar t, \lbar \gamma)} \sub \RV^{k-1} \times \Gamma^{m-k+1}.
\]
is nonempty and definable. We may repeat this procedure with respect to $B_1$ and get a definable subset $B_2 \sub \RV^{k-2} \times \Gamma^{m-k+2}$, and so on. Eventually we obtain a nonempty definable subset $B_k \sub \Gamma^{m}$ with the following property: if $\lbar \gamma \in B_k$ then there is a $(t_1, \ldots, t_k, \ldots, t_m) \in \vrv^{-1}(\lbar \gamma) \cap B$ such that $t_1, \ldots, t_k$ are algebraically independent and hence $\dim_{\RV}(\vrv^{-1}(\lbar \gamma) \cap B) = k$. Now, since $\Gamma$ is $o$-minimal, some $\lbar \gamma \in B_k$ is definable.
\end{proof}

\begin{defn}
The \emph{$\RV$-fiber dimension} of $A$, denoted by $\dim^{\fib}_{\RV}(A)$, is
\[
\max \set{\dim_{\RV} (\fib(A, \lbar a)) : \lbar a \in \pvf(A)}.
\]
\end{defn}

\begin{lem}\label{RV:fiber:dim:same}
Suppose that $f : A \fun A'$ is a definable bijection. Then $\dim^{\fib}_{\RV}(A) = \dim^{\fib}_{\RV} (A')$.
\end{lem}
\begin{proof}
Let $\dim^{\fib}_{\RV}(A) = k_1$ and $\dim^{\fib}_{\RV}(A') =
k_2$. Since for every $\lbar b \in \pvf(A')$ there is a
$\lbar b$-definable finite-to-one function $h_{\lbar b} :
\fib(A', \lbar b) \fun \RV^{k_2}$, by compactness, there
is a definable function $h : A' \fun \RV^{k_2}$ such that $h
\rest \fib(A', \lbar b)$ is finite-to-one for every $\lbar
b \in \pvf (A')$. For every $\lbar a \in \pvf (A)$, by
Corollary~\ref{function:rv:to:vf:finite:image}, the subset $(\pvf
\circ f)(\fib(A, \lbar a))$ is finite. So the function
$g_{\lbar a}$ on $\fib(A, \lbar a)$ given by
\[
(\lbar a, \lbar t) \efun (h \circ f)(\lbar a, \lbar t)
\]
is $\lbar a$-definable and finite-to-one. So $k_1 \leq k_2$. Symmetrically
we also have $k_1 \geq k_2$ and hence $k_1 = k_2$.
\end{proof}

\subsection{Categories of definable subsets}

The class of objects and the class of morphisms of any
category $\mathcal{C}$ are denoted by $\ob \mathcal{C}$ and $\mor
\mathcal{C}$, respectively. By $A \in \mdl C$ we usually mean that $A$ is an object of $\mdl C$.

\begin{defn}[$\VF$-categories]\label{defn:VF:cat}
The objects of the category $\VF[k, \cdot]$ are the definable
subsets of $\VF$-dimension $\leq k$. The morphisms in this
category are the definable functions between the objects.

The category $\VF[k]$ is the full subcategory of $\VF[k, \cdot]$
of the definable subsets that have $\RV$-fiber dimension 0 (that is, all the $\RV$-fibers are finite). The category $\VF_*[\cdot]$ is
the union of the categories $\VF[k, \cdot]$. The category $\VF_*$
is the union of the categories $\VF[k]$.
\end{defn}

Note that, for any definable subset $A$, by Lemma~\ref{finite:VF:project:RV} and Lemma~\ref{dim:VF:pullback}, $\fib(A, \lbar t)$ is finite for every $\lbar t \in \prv(A)$ if and only if $A \in \VF[0, \cdot]$. Also, by Lemma~\ref{RV:fiber:dim:same}, $A \in \VF[k]$ if and only if there is a definable finite-to-one map $A \fun \VF^k$.

\begin{defn}\label{defn:weight}
For any tuple $\lbar t = (t_1, \ldots, t_n) \in \RV$, the
\emph{weight} of $\lbar t$ is the number $\abs{\set{i \leq n : t_i
\neq \infty}}$, which is denoted by $\wgt(\lbar t)$.
\end{defn}

\begin{defn}[$\RV$-categories]\label{defn:RV:cat}
An object of the category $\RV[k,\cdot]$ is a definable pair $(U,f)$, where $U \sub \RV^m$ for some $m$ and $f: U \fun \RV^k$ is a function ($\RV^0$ is taken to be the singleton $\{\infty\}$). We often denote the projections $\pr_i \circ f$ as $f_{i}$ and write $f$ as $(f_{1}, \ldots, f_{k})$. The \emph{companion
$U_f$} of $(U,f)$ is the subset $\set{(f(\lbar u), \lbar u) :
\lbar u \in U}$.

For any two objects $(U,f), (U',f')$ in $\RV[k,\cdot]$ and any function $F: U \fun U'$, if $\wgt (f(\lbar u)) \leq \wgt ((f' \circ F)(\lbar u))$ for every $\lbar u \in U$ then we say that $F$ is \emph{volumetric}. If
$F$ is definable, volumetric, and, for every $\lbar t \in \RV^k$ the subset $(f' \circ F)(f^{-1}(\lbar t))$ is finite, then it is a morphism in $\mor \RV[k,\cdot]$.

The category $\RV[k]$ is the full subcategory of $\RV[k,\cdot]$ of
the pairs $(U,f)$ such that $f: U \fun \RV^k$ is finite-to-one.

Direct sums (coproducts) over these categories are formed naturally:
\[
\RV[\leq i, \cdot] = \coprod_{0 \leq k \leq i} \RV[k, \cdot],\,\, \RV[*, \cdot] = \coprod_{0 \leq k} \RV[k, \cdot],
\]
and similarly for $\RV[\leq i]$ and $\RV[*]$.
\end{defn}

We usually just write $A$ for the object $(A, \id) \in \RV[k,\cdot]$. Also, for any object in $\RV[k,\cdot]$ of the form $(U, \pr_E)$, we may assume that $(U, \pr_E)$ is $(U, \pr_{\leq k})$ if this is more convenient. This should not cause any confusion in context.

One of the main reasons for the peculiar forms of the objects and
the morphisms in the $\RV$-categories is that each isomorphism
class in these categories may be ``lifted'' to an isomorphism
class in the corresponding $\VF$-category. See
Proposition~\ref{RV:iso:class:lifted} and
Corollary~\ref{L:semigroup:hom} for details.

A \emph{subobject} of an object $A$ of a $\VF$-category is just a
definable subset. A \emph{subobject} of an object $(U, f)$ of an
$\RV$-category is a definable pair $(A, g)$ with $A$ a subset of $U$ and $g = f \rest A$. Note that the inclusion map is
a morphism in both cases.

Notice that the cartesian product of two objects $A$, $B \in \VF[k, \cdot]$ may or may not be in $\VF[k, \cdot]$. On the other
hand, the cartesian product of two objects $(U,f)$, $(U',f') \in
\RV[k,\cdot]$ is the object $(U \times U', f \times f') \in
\RV[2k,\cdot]$, which is definitely not in $\RV[k,\cdot]$ if $k >
0$. Hence, in $\RV[*, \cdot]$ or $\RV[*]$, multiplying with a
singleton in general changes isomorphism class.

The categories $\VF_*[\cdot]$ and $\VF_*$ are formed through union
instead of direct sum or other means that induces more complicated
structure. The reason for this is that the main goal of the
Hrushovski-Kazhdan integration theory is to assign motivic
volumes, that is, elements in the Grothendieck groups of the
$\RV$-categories, to the definable subsets, or rather, the
isomorphism classes of the definable subsets, in the
$\VF$-categories, and the simplest categories that contain all the
definable subsets that may be ``measured'' in this motivic way are
$\VF_*[\cdot]$ and $\VF_*$. In contrast, the unions of the
$\RV$-categories are naturally endowed with the structure of
direct sum, which gives rise to graded Grothendieck semirings. The ring homomorphisms are obtained by ``passing to the limit''. These will be made precise in a sequel.

\begin{defn}
For any $(U,f) \in \RV[k,\cdot]$ and any $F \in \mor
\RV[k,\cdot]$, let $\bb E_k(f)$ be the function on $U$ given by
$\lbar u \efun (f(\lbar u), \infty)$, $\bb E_k(U, f) = (U, \bb
E_k(f))$, and $\bb E_k(F) = F$. Obviously
\[
\bb E_k : \RV[k,\cdot] \fun  \RV[k+1,\cdot]
\]
is a functor that is faithful, full, and
injective on objects. For any $i < j$ let $\bb E_{i,j} = \bb
E_{j-1} \circ \cdots \circ \bb E_{i}$ and $\bb E_{i,i} = \id$.
\end{defn}

Homomorphisms between Grothendieck groups shall be induced by the following fundamental maps:

\begin{defn}\label{def:L}
For any $(U,f) \in \RV[k,\cdot]$, let
\[
\mathbb{L}_k(U,f) = \bigcup \set{\rv^{-1}(f(\lbar u)) \times
\set{\lbar u}: \lbar u \in U}.
\]
The map $\mathbb{L}_k: \ob \RV[k,\cdot] \fun \ob \VF[k,\cdot]$ is
called the \emph{$k$th canonical $\RV$-lift}. The map
$\mathbb{L}_{\leq k}: \ob \RV[\leq k,\cdot] \fun \ob \VF[k,\cdot]$
is given by
\[
((U_1, f_1), \ldots, (U_k, f_k)) \efun \biguplus_{i \leq k} (\bb L_k \circ
\bb E_{i, k})(U_i, f_i).
\]
The map $\mathbb{L}: \ob \RV[*,\cdot] \fun \ob \VF_*[\cdot]$ is
simply the union of the maps $\mathbb{L}_{\leq k}$.
\end{defn}

For notational convenience, when there is no danger of confusion,
we shall drop the subscripts and simply write $\bb E$ and
$\mathbb{L}$ for these maps.

Observe that if $(U,f) \in \RV[k]$ then $\mathbb{L}(U, f) \in \VF[k]$ and hence the restriction $\mathbb{L}: \ob \RV[k] \fun
\ob \VF[k]$ is well-defined. Similarly we have the maps
\[
\mathbb{L}: \ob \RV[\leq k] \fun \ob \VF[k],\quad
\mathbb{L}: \ob \RV[*] \fun \ob \VF_*.
\]
Also note that $\rv(\mathbb{L}(U, f)) = U_f$ for $(U,f) \in \RV[k,\cdot]$.

For any two objects $(U,f), (U',f') \in \RV[k,\cdot]$ and any definable function $F : U \fun U'$ there is a naturally induced function $F_{f,f'} : U_f \fun U'_{f'}$ given by
\[
(f(\lbar u), \lbar u) \efun ((f' \circ F)(\lbar u), F(\lbar u)).
\]
We have:

\begin{lem}\label{RV:morphism:condition}
Suppose that $F$ is volumetric and there is a definable
function $F^{\uparrow} : \mathbb{L}(U, f) \fun \mathbb{L}(U', f')$
such that the diagram
\[
\bfig
  \square(0,0)/->`->`->`->/<600,400>[\mathbb{L}(U, f)`U_f`\mathbb{L}(U', f')`
   U'_{f'}; \rv`F^{\uparrow}`F_{f, f'}`\rv]
 \square(600,0)/->`->`->`->/<600,400>[U_f`U`U'_{f'}`U';
  \pr_{>k}``F`\pr_{>k}]
\efig
\]
commutes. Then $F$ is a morphism in $\RV[k,\cdot]$.
\end{lem}
\begin{proof}
It is enough to show that, for every $\lbar u \in U$ and every $i \leq k$,
\[
(f'_i \circ F)(\lbar u) \in \acl(f(\lbar u)),
\]
which is equivalent to $(\pr_i \circ F_{f, f'})(f(\lbar u), \lbar u) \in \acl(f(\lbar u))$. To that end, fix a $\lbar u \in U$. Let $\lbar a \in \rv^{-1}(f(\lbar u))$ and $F^{\uparrow}(\lbar a, \lbar u) = (b_1, \ldots, b_k, \lbar u')$. By Lemma~\ref{dcl:to:ac}, $b_i \in \acl(\lbar a)$ and hence
\[
(\pr_i \circ F_{f, f'})(f(\lbar u), \lbar u) = \rv(b_i) \in \acl(\lbar a)
\]
for each $i \leq k$. By Corollary~\ref{acl:VF:transfer:acl:RV}, $\rv(b_i) \in \acl(f(\lbar u))$.
\end{proof}

\begin{rem}\label{RV:isomorphism:weaker:condition}
In Lemma~\ref{RV:morphism:condition}, if both $F$ and
$F^{\uparrow}$ are bijections then we may drop the assumption that
$F$ is volumetric, since it is guaranteed by the commutative
diagram and Corollary~\ref{VF:dim:rv:polydisc}.
\end{rem}

\subsection{Grothendieck groups}

We now introduce the Grothendieck groups associated with the
categories defined above. The construction is of course the same
for any reasonable category of definable sets of a first-order
theory. For concreteness, we shall limit our attention to the
present context.


Let $\mathcal{C}$ be a $\VF$-category or an $\RV$-category. For
any $A \in \ob \mathcal{C}$, let $[A]$ denote the isomorphism
class of $A$. The \emph{Grothendieck semigroup} of $\mathcal{C}$,
denoted by $\gsk\mathcal{C}$, is the semigroup generated by the
isomorphism classes $[A]$ of $\mathcal{C}$, subject to the relation
\[
[A] + [B] = [A \cup B] + [A \cap B].
\]
It is easy to check that $\gsk\mathcal{C}$ is actually a
commutative monoid, the identity element being $[\0]$ or $([\0],
\ldots)$. Since $\mathcal{C}$ always has disjoint unions, the
elements of $\gsk\mathcal{C}$ are precisely the isomorphism
classes of $\mathcal{C}$. If $\mathcal{C}$ is one of the
categories $\VF_*[\cdot]$, $\VF_*$, $\RV[*,\cdot]$, and $\RV[*]$
then it is closed under cartesian product. In this case,
$\gsk\mathcal{C}$ has a semiring structure with multiplication
given by
\[
[A][B] =[A \times B].
\]
Since the symmetry isomorphisms $A \times B \fun B \times A$ and
the association isomorphisms $(A \times B) \times C \fun A \times
(B \times C)$ are always present in these categories,
$\gsk\mathcal{C}$ is always a commutative semiring.

\begin{rem}
If $\mathcal{C}$ is either $\VF_*[\cdot]$ or $\VF_*$ then the
isomorphism class of definable singletons is the multiplicative
identity of $\gsk\mathcal{C}$. If $\mathcal{C}$ is
$\RV[*,\cdot]$ then we adjust multiplication when $\RV[0, \cdot]$
is involved as follows. For any $(U, f) \in \RV[0, \cdot]$ and
$(V, g) \in \RV[k, \cdot]$, let
\[
[(U, f)][(X, g)] = [(X, g)][(U, f)] = [(U \times V, g^*)],
\]
where $g^*$ is the function on $U \times V$ given by $(\lbar t, \lbar s) \efun g(\lbar s)$.
It is easily seen that, with this adjustment, $\gsk \RV[*,\cdot]$
becomes a filtrated semiring and its multiplicative identity
element is the isomorphism class of $(\infty, \id)$ in $\RV[0,
\cdot]$. Multiplication in $\gsk \RV[*]$ is adjusted in the same
way.
\end{rem}

\begin{defn}
A \emph{semigroup congruence relation} on $\gsk \mathcal C$ is a
sub-semigroup $R$ of the semigroup $\gsk \mathcal C \times \gsk
\mathcal C$ such that $R$ is an equivalence relation on $\gsk
\mathcal C$. Similarly, a \emph{semiring congruence relation} on
$\gsk \mathcal C$ is a sub-semiring $R$ of the semiring $\gsk
\mathcal C \times \gsk \mathcal C$ such that $R$ is an equivalence
relation on $\gsk \mathcal C$.
\end{defn}

Let $R$ be a semigroup congruence relation on $\gsk \mathcal C$
and $(x, y), (v, w) \in R$. Then $(x + v, y + v)$, $(y + v, y+w)
\in R$ and hence $(x+v, y+w) \in R$. Therefore the equivalence
classes of $R$ has a semigroup structure induced by that
of $\gsk \mathcal C$. This semigroup is denoted by $\gsk \mathcal
C / R$ and is also referred to as a Grothendieck semigroup.
Similarly, if $R$ is a semiring congruence relation on $\gsk
\mathcal C$ then $\gsk \mathcal C / R$ is actually a Grothendieck
semiring.

\begin{rem}\label{remark:semiring:cong}
Let $R$ be an equivalence relation on the semiring $\gsk \mathcal
C$. If for every $(x, y) \in R$ and every $z \in \gsk \mathcal C$
we have $(x +z, y+z) \in R$ and $(xz, yz) \in R$ then $R$ is
a semiring congruence relation.
\end{rem}

Let $(\Z^{\gsk \mathcal{C}}, \oplus)$ be the free abelian group
generated by the elements of $\gsk \mathcal{C}$ and $C$ the
subgroup of $(\Z^{\gsk \mathcal{C}}, \oplus)$ generated by all
elements of $(\Z^{\gsk \mathcal{C}}, \oplus)$ of the types
\[
(1\cdot x) \oplus ((-1) \cdot x), \quad
(1\cdot x) \oplus (1 \cdot y) \oplus ((-1) \cdot (x + y)),
\]
where $x, y \in \gsk \mathcal{C}$. The \emph{Grothendieck group}
of $\mathcal C$, denoted by $\ggk \mathcal C$, is the formal
groupification $(\Z^{(\gsk \mathcal{C})}, \oplus) / C$ of $\gsk
\mathcal C$, which is essentially unique by the universal mapping
property. In general the natural homomorphism from $\gsk \mathcal C$ into $\ggk \mathcal C$ is not injective. Note that if $\gsk \mathcal C$ is a commutative semiring then $\ggk \mathcal C$ is naturally a commutative ring.

It is easily checked that $\bb E_k$ induces an injective semigroup
homomorphism
\[
\gsk \RV[k,\cdot] \fun \gsk \RV[k+1,\cdot],
\]
which is also denoted by $\bb E_k$.


\section{$\RV$-pullbacks and special bijections}\label{section:RV:product}

We shall adopt~\cite[Convention~4.20]{Yin:QE:ACVF:min}: Since definably bijective subsets are to be identified, for a subset $A$, we shall tacitly substitute its canonical image $\can(A)$ for it in the discussion if it is necessary or is just more convenient.

For any subset $U$, recall from~\cite[Definition~4.21]{Yin:QE:ACVF:min} that the \emph{$\RV$-hull} of $U$ is the union of the $\rv$-polydiscs that have a nonempty intersection with $U$. If $U$ is equal to its $\RV$-hull then $U$ is an \emph{$\RV$-pullback}. An $\RV$-pullback is \emph{degenerate} if it contains a degenerate $\rv$-polydisc and is \emph{strictly degenerate} if it only contains degenerate $\rv$-polydiscs.

Here comes the general version of~\cite[Definition~4.22]{Yin:QE:ACVF:min}:

\begin{defn}\label{defn:special:bijection}
Let $A \sub \VF \times \VF^n \times \RV^m$. Let $C \sub \RVH(A)$
be an $\RV$-pullback and $\lambda: \pr_{>1}(C \cap A) \fun \VF$ a function
such that every $(\lambda(\lbar a_1, \lbar t), \lbar a, \lbar t)$ is in $C$. Let
\begin{gather*}
C^{\sharp} = \bigcup_{(\lbar a_1, t_1, \lbar t_1) \in \pr_{>1} C} \bigl( \bigl(\bigcup \set{\rv^{-1}(t): \vrv(t) > \vrv(t_1)} \bigr) \times
\set{(\lbar a_1, t_1, \lbar t_1)} \bigr),\\
\RVH(A)^{\sharp} = C^{\sharp} \uplus (\RVH(A) \mi C).
\end{gather*}
The \emph{centripetal transformation $\eta : A \fun \RVH(A)^{\sharp}$ with respect to $\lambda$} is defined by
\[
\begin{cases}
  \eta (a_1, \lbar a_1, \lbar t) = (a_1 - \lambda(\lbar a_1, \lbar t), \lbar a_1, \lbar t), & \text{on } C \cap A,\\
  \eta = \id, & \text{on } A \mi C.
\end{cases}
\]
Note that $\eta$ is injective. The inverse of $\eta$ is naturally called the
\emph{centrifugal transformation with respect to $\lambda$}. The
function $\lambda$ is called a \emph{focus map of $X$}. The $\RV$-pullback $C$ is
called the \emph{locus} of $\lambda$. A \emph{special bijection}
$T$ is an alternating composition of centripetal transformations
and the canonical bijection. The \emph{length} of a special
bijection $T$, denoted by $\lh T$, is the number of centripetal
transformations in $T$. The image $T(A)$ is sometimes denoted by $A^{\sharp}$.
\end{defn}

Note that we should have included the index of the targeted
$\VF$-coordinate as a part of the data of a focus map. Since it
should not cause confusion in context, we shall suppress mentioning it
for notational ease.

We shall only be concerned with definable special bijections.

Clearly if $A$ is an $\RV$-pullback and $T$ is a special bijection on $A$ then $T(A)$ is an $\RV$-pullback. Recall that a subset $A$ is called a \emph{deformed $\RV$-pullback} if there is a special bijection $T$ such that $T(A)$ is an $\RV$-pullback.

\begin{lem}\label{simplex:with:hole:rvproduct}
Every definable subset $A \sub \VF \times \RV^m$ is a deformed $\RV$-pullback.
\end{lem}
\begin{proof}
See~\cite[Lemma~4.26]{Yin:QE:ACVF:min}
\end{proof}

\begin{rem}\label{preimage:of:special}
Let $A \sub \VF \times \RV^m$ be a deformed $\RV$-pullback and $T : A \fun U$ a special bijection that witnesses this. By a routine induction, we see that if $\rv^{-1}(s) \times \set{(s, \lbar t)} \sub U$ with $s \neq \infty$ then $T^{-1}(\rv^{-1}(s) \times \set{(s, \lbar t)})$ is an open polydisc that is contained in an $\rv$-polydisc.
\end{rem}

Let $f : A \fun B$ be a function. We say that $f$ is \emph{contractible} if for every $\rv$-polydisc $\gp \sub \RVH(A)$ the subset $f(\gp \cap A)$ is contained in one $\rv$-polydisc. Clearly, if $f : A \fun B$ is a (definable) contractible function then there is a unique (definable) function $f_{\downarrow} : \rv(A) \fun \rv(B)$ such that the diagram commutes:
\[
\bfig
  \square/->`->`->`->/<600,400>[A`B`\rv(A)`\rv(B);  f`\rv`\rv`f_{\downarrow}]
 \efig
\]
In this case we say that $f_{\downarrow}$ is the \emph{contraction} of $f$.

The following technical result is a major tool for the Hrushovski-Kazhdan construction as presented in~\cite{yin:hk:part:1}.

\begin{thm}\label{special:bi:polynomial:constant:disc}
Let $F(\lbar X) = F(X_1, \ldots, X_n)$ be a polynomial with coefficients in $\VF(S)$, $\lbar u \in \RV^n$ a definable tuple, $\tau : \rv^{-1}(\lbar u) \fun A$ a special bijection, and $f = F \circ \tau^{-1}$. Then there is a special bijection $T$ on $A$ such that $f \circ T^{-1}$ is contractible.
\end{thm}
\begin{proof}
First observe that if the assertion holds for one polynomial $F(\lbar X)$ then it holds simultaneously for any finite number of polynomials. We do induction on $n$. For the base case $n=1$, we simply write $X$ for $\lbar X$. Let $T$ be a special bijection on $A$. For any $\rv$-polydisc $\gp \sub T(A)$, let $k_T(\gp)$ be the size of the set $\{\lbar x \in \gp : (f \circ T^{-1})(\lbar x) = 0 \}$.
\begin{incla}
There is a special bijection $T^*$ on $T(A)$ such that $f \circ (T^* \circ T)^{-1}$ is contractible.
\end{incla}
\begin{proof}
By compactness, we may concentrate on one $\rv$-polydisc $\gp = \rv^{-1}(s) \times \set{(s, \lbar r)} \sub T(A)$. We do induction on $k_T(\gp)$. For the base case $k_T(\gp) = 1$, consider the focus map $\lambda : \{(s, \lbar r)\} \fun \VF$ such that $f(T^{-1}(\lambda(s, \lbar r), s, \lbar r)) = 0$ and the special bijection $T^*$ on $\gp$ given by
\[
(b, s, \lbar r) \efun (b - \lambda(s, \lbar r), \rv(b - \lambda(s, \lbar r)), s, \lbar r).
\]
By Remark~\ref{preimage:of:special}, for every $\rv$-polydisc $\gr \sub T^*(\gp)$, $(T^* \circ T \circ \tau)^{-1}(\gr)$ is either the root of $F(X)$ in question or an open ball that contains no roots of any $F(X)$. So $T^*$ is as required.

For the inductive step $k_T(\gp) = m > 1$, let $(d_1, s, \lbar r), \ldots, (d_m, s, \lbar r) \in \gp$ be the points in question and $d$ the average of $d_1, \ldots, d_m$. Consider the special bijection $T^*$ on $\gp$ given by
\[
(b, s, \lbar r) \efun (b - d, \rv(b - d), s, \lbar r).
\]
By Lemma~\ref{average:0:rv:not:constant}, $\rv$ is not
constant on $\set{d_1 - d, \ldots, d_m - d}$ and hence
$k_{T^* \circ T}(\gr) < m$ for every $\rv$-polydisc $\gr \sub T^*(\gp)$. So we are done by compactness and the inductive hypothesis.
\end{proof}

This completes the base case of the induction.

We now proceed to the inductive step. As above, we may concentrate on one $\rv$-polydisc $\gp = \rv^{-1}(\lbar s) \times \set{(\lbar s, \lbar r)} \sub A$. Let $\phi(\lbar X, Y)$ be a quantifier-free formula that defines the function $(\rv \circ f) \rest \gp$, where $Y$ is the free $\RV$-sort variable. Let $G_i(\lbar X)$ enumerate the occurring polynomials of $\phi(\lbar X, Y)$. For each $a \in \rv^{-1}(s_1)$ let $G_{i,a} = G_i(a, X_2, \ldots, X_n)$. By the inductive hypothesis, there is a special bijection $R_{a}$ on $\rv^{-1}(s_2, \ldots, s_n)$ such that every function $G_{i,a} \circ R_a^{-1}$ is contractible. Let $U_{j, a}$ enumerate the loci used in $R_{a}$ and $\lambda_{j, a}$ the corresponding focus maps. By compactness,
\begin{enumerate}
  \item for each $i$ there is a quantifier-free formula $\psi_i(X_1, \lbar Z', Z)$ such that $\psi_i(a, \lbar Z', Z)$ defines the contraction of $G_{i,a} \circ R_a^{-1}$,
  \item there is a quantifier-free formula $\theta(X_1, \lbar Z'')$ such that $\theta(a, \lbar Z'')$ determines the sequence $\rv(U_{j, a})$ and the $\VF$-coordinates targeted by $\lambda_{j, a}$.
\end{enumerate}
Let $H_{k}(X_1)$ enumerate the occurring polynomials of the formulas $\psi_i(X_1, \lbar Z', Z)$, $\theta(X_1, \lbar Z'')$. Applying the inductive hypothesis again, we obtain a special bijection $T_1$ on $\rv^{-1}(s_1)$ such that every function $H_{k} \circ T_1^{-1}$ is contractible. This means that, for every $\rv$-polydisc $\gq \sub T_1(\rv^{-1}(s_1))$ and every $a_1, a_2 \in T_1^{-1}(\gq)$, the formulas $\psi_i(a_1, \lbar W, Z)$, $\psi_i(a_2, \lbar W, Z)$ define the same function and the special bijections $R_{a_1}$, $R_{a_2}$ may be naturally glued together to form one special bijection on $\{a_1, a_2\} \times \rv^{-1}(s_2, \ldots, s_n)$. Consequently, $T_1$ and $R_{a}$ naturally induce a special bijection $T$ on $\gp$ such that each function $G_i \circ T^{-1}$ is contractible. This implies that $f \circ T^{-1}$ is contractible.
\end{proof}

We immediately give a slightly more general version of Theorem~\ref{special:bi:polynomial:constant:disc}, which is easier to use:

\begin{thm}\label{special:bi:polynomial:constant}
Let $F(\lbar X) = F(X_1, \ldots, X_n)$ be a polynomial with coefficients in $\VF(S)$, $B \sub \VF^n$ a definable subset, $\tau : B \fun A$ a special bijection, and $f = F \circ \tau^{-1}$. Then there is a special bijection $T$ on $A$ such that $T(A)$ is an $\RV$-pullback and $f \circ T^{-1}$ is contractible.
\end{thm}
\begin{proof}
By compactness, we may concentrate on a subset of the form $A_{\gp} = \gp \cap A$, where $\gp$ is an $\rv$-polydisc. Let $\phi(\lbar X, Z)$ be a quantifier-free formula that defines the function $(\rv \circ f) \rest A_{\gp}$. Let $F_i(\lbar X)$ enumerate the occurring polynomials of $\phi(\lbar X, Z)$. By Theorem~\ref{special:bi:polynomial:constant:disc} there is a special bijection $T$ on $\gp$ such that each function $F_i \circ T^{-1}$ is contractible. This means that, for each $\rv$-polydisc $\gq \sub T(\gp)$,
\begin{enumerate}
  \item either $T^{-1}(\gq) \sub A_{\gp}$ or $T^{-1}(\gq) \cap A_{\gp} = \0$,
  \item if $T^{-1}(\gq) \sub A_{\gp}$ then $(\rv \circ f \circ T^{-1})(\gq)$ is a singleton.
\end{enumerate}
So $T \rest A_{\gp}$ is as required.
\end{proof}

Now Lemma~\ref{simplex:with:hole:rvproduct} may be easily generalized to all dimensions:

\begin{cor}\label{all:subsets:rvproduct}
Every definable subset $A \sub \VF^n \times \RV^m$ is a definable
deformed $\RV$-pullback.
\end{cor}
\begin{proof}
By compactness, we may assume that $A$ is contained in an $\rv$-polydisc. Then the assertion simply follows from Theorem~\ref{special:bi:polynomial:constant}.
\end{proof}

Applying Lemma~\ref{dim:RV:fibers} and Lemma~\ref{RV:fiber:dim:same}, we get the following corollary:

\begin{cor}\label{L:surjective}
The map $\mathbb{L}: \ob \RV[k, \cdot] \fun \ob \VF[k, \cdot]$ is
surjective on the isomorphism classes of $\VF[k, \cdot]$. The map $\mathbb{L}: \ob \RV[k] \fun \ob \VF[k]$ is surjective on the isomorphism classes of $\VF[k]$.
\end{cor}

\section{Interlude: quantifier elimination for henselian fields}\label{section:inter}

The analysis on special transformations in Section~\ref{section:RV:product} leads to a general quantifier elimination result for henselian fields. Pas's quantifier elimination result {\cite[Theorem~4.1]{Pa89} may be recovered from it.

\begin{defn}
A substructure $M$ is \emph{functionally closed} if, for any definable subset $A$ and any definable function $f$ on $A$, $f(A \cap M) \sub M$.
\end{defn}

\begin{lem}\label{cut:to:hensel:substru}
Let $M$ be a substructure such that $(\VF(M), \OO(M))$ is a nontrivially valued henselian field and $\rv(\VF(M)) = \RV(M)$. Then $M$ is functionally closed.
\end{lem}
\begin{proof}
By Corollary~\ref{aclS:model}, $\acl(M) \models \ACVF_S(0,0)$. Note that $\VF(\acl(M)) = \VF(M)^{\alg}$. Since the valued field $(\VF(M), \OO(M))$ is henselian, it is the fixed field under the valued field automorphisms of $(\VF(M)^{\alg}, \OO(\acl(M)))$ over $(\VF(M), \OO(M))$. On the other hand, these valued field automorphisms are in one-to-one correspondence with the $\lan{RV}$-automorphisms of $\acl(M)$ over $M$. So $\VF (\dcl (M)) = \VF(M)$. Since $M$ is $\VF$-generated, by Lemma~\ref{algebraic:balls:definable:centers}, every $t \in \RV(\dcl(M))$ has an $M$-definable point in $\VF$. So $M = \dcl(M)$ and the lemma follows.
\end{proof}

Let $\hen_S(0, 0)$ be the theory of henselian fields of pure characteristic $0$ in a language $\lan{H}$ that expands $\lan{RV}$, where the expansion happens only in the $\RV$-sort. Such a theory may be formulated as in Definition~\ref{defn:acvf}, with obvious modifications. Note that $\hen_S(0, 0)$ includes the statement that the function $\rv$ is surjective.


\begin{lem}
Let $\phi(\lbar X)$ be a $\VF$-quantifier-free $\lan{H}$-formula, where $\lbar X = (X_1, \ldots, X_n)$ are the free $\VF$-sort variables. Then $\hen_S(0,0)$ proves that $\ex{\lbar X} \phi(\lbar X)$ is equivalent to a $\VF$-quantifier-free $\lan{H}$-formula.
\end{lem}
\begin{proof}
Without loss of generality, we may assume that $\phi(\lbar X)$ contains no $\VF$-sort literals. Let $F_i(\lbar X)$ be the occurring polynomials of $\phi(\lbar X, \lbar Y)$. Let $\phi^*(\lbar Z)$ be the formula obtained from $\phi(\lbar X)$ by replacing each term $\rv(F_i(\lbar X))$ with a new $\RV$-sort variable $Z_i$. Let $M \models \hen_S(0,0)$ such that its reduct to $\lan{RV}$ is a substructure of $\gC$.

By Theorem~\ref{special:bi:polynomial:constant}, there is an $\RV$-pullback $A$ and an $\lan{RV}$-definable bijection $T : A \fun \VF^n$ such that, for every $\rv$-polydisc $\gp \sub A$, every subset of the form $\rv(F_i(T(\gp)))$ is a singleton. This induces functions $f_{i} : \rv(A) \fun \RV$, defined by quantifier-free $\lan{RV}$-formulas $\psi_i(\lbar Y, Z_i)$ (hence no $\VF$-sort quantifiers). By Lemma~\ref{cut:to:hensel:substru}, $T^{-1} (M \cap \VF^n) \sub M$ and hence $T^{-1} (M \cap \VF^n) = A \cap M$. Similarly $f_i(\rv(A \cap M)) \sub M$ for every $i$. Therefore
\[
M \models \ex{\lbar X} \phi(\lbar X) \liff \ex{\lbar Y, \lbar Z} \bigl( \bigwedge_i \psi_i(\lbar Y, Z_i) \wedge \phi^*(\lbar Z) \bigr).
\]
The lemma follows.
\end{proof}

By elementary logic this lemma yields:

\begin{prop}\label{qe:hensel}
The theory $\hen_S(0,0)$ admits elimination of $\VF$-quantifiers.
\end{prop}

If angular component map exists then $\RV^{\times}$ may be understood as $\K^{\times} \oplus \Gamma$. Hence we have the following:

\begin{cor}[{\cite[Theorem~4.1]{Pa89}}]
The theory of henselian fields of pure characteristic $0$ in any Denef-Pas language admits elimination of field sort quantifiers.
\end{cor}

The ``descent'' technique in this section can also be applied to theories of henselian fields with sections, which are formulated in a natural way as in~\cite{Yin:QE:ACVF:min}. This will be explained elsewhere.

\section{Lifting functions from $\RV$ to $\VF$}\label{section:lifting}

We shall show in this section that the map $\mathbb{L}$ actually induces
homomorphisms between various Grothendieck semigroups when $S$ is
a $(\VF, \Gamma)$-generated substructure.

Any polynomial in $\OO[\lbar X]$ corresponds to a polynomial in
$\K[\lbar X]$ via the canonical quotient map. The following
definition generalizes this phenomenon.

\begin{defn}\label{def:gamma:poly}
Let $\lbar \gamma = (\gamma_1, \ldots, \gamma_n) \in \Gamma$. A
polynomial $F(\lbar X) = \sum_{\lbar i j} a_{\lbar i j} \lbar
X^{\lbar i}$ with coefficients $a_{\lbar i j} \in \VF$ is a
\emph{$\lbar \gamma$-polynomial} if there is an $\alpha \in
\Gamma$ such that
\[
\alpha = \vv(a_{\lbar i j}) + i_1 \gamma_1 + \cdots + i_n \gamma_n
\]
for each $\lbar i j =(i_1, \ldots, i_n, j)$. In this case we say that $\alpha$ is a
\emph{residue value} of $F(\lbar X)$ (with respect to $\lbar
\gamma$). For a $\lbar \gamma$-polynomial $F(\lbar X)$ with
residue value $\alpha$ and a $\lbar t \in \RV$ with $\vrv(\lbar t)
= \lbar \gamma$, if $\vv (F(\lbar a)) > \alpha$ for some (hence all) $\lbar a \in
\rv^{-1}(\lbar t)$ then $\lbar t$ is a \emph{residue root} of
$F(\lbar X)$.

If $\lbar t \in \RV$ is a common residue root of the
$\lbar \gamma$-polynomials $F_1(\lbar X), \ldots, F_n(\lbar X)$
but is not a residue root of the $\lbar \gamma$-polynomial
\[
\det \partial (F_1, \ldots, F_n) / \partial \lbar X,
\]
then we say that $F_1(\lbar X), \ldots, F_n(\lbar X)$ are
\emph{minimal} for $\lbar t$ and $\lbar t$ is a \emph{simple}
common residue root of $F_1(\lbar X), \ldots, F_n(\lbar X)$.
\end{defn}

Therefore, according to this definition, every polynomial in
$\K[\lbar X]$ is the projection of a $\lbar 0$-polynomial $F(\lbar X)$ with residue value $0$, where $\lbar 0 = (0, \ldots, 0)$.

Hensel's lemma is accordingly generalized as follows.

\begin{lem}[Generalized Hensel's lemma]\label{hensel:lemma}
Let $F_1(\lbar X), \ldots, F_n(\lbar X)$ be $\lbar \gamma$-polynomials with residue values $\alpha_1, \ldots, \alpha_n$, where $\lbar \gamma = (\gamma_1, \ldots, \gamma_n) \in \Gamma$. For every simple common residue root $\lbar t = (t_1,
\ldots, t_n) \in \RV$ of $F_1(\lbar X), \ldots, F_n(\lbar X)$
there is a unique $\lbar a \in \rv^{-1}(\lbar t)$ such that $F_i(\lbar a) = 0$ for every $i$.
\end{lem}
\begin{proof}
Without loss of generality we may work in a topologically complete submodel of $\ACVF$ of rank $1$.

Fix a simple common residue root $\lbar t = (t_1, \ldots, t_n) \in
\RV$ of $F_1(\lbar X), \ldots, F_n(\lbar X)$. Choose a $c_i \in
\rv^{-1}(t_i)$. Changing the coefficients accordingly we may
rewrite each $F_i(\lbar X)$ as $F_i(X_1/c_1, \ldots, X_n/c_n)$.
Write $Y_i$ for $X_i/c_i$. Note that, for each $i$, the coefficients of the $\lbar 0$-polynomial $F_i(\lbar Y)$ are all of the same value $\alpha_i$. For each $i$ choose an $e_i \in \VF$ with $\vv(e_i) = - \alpha_i$. We have that each $\lbar 0$-polynomial $F_i^*(\lbar Y) = e_iF_i(\lbar Y)$ has residue
value $0$ (that is, the coefficients of $F_i^*(\lbar Y)$ is of
value $0$). Clearly $(1, \ldots, 1)$ is a common residue root of $F_1^*(\lbar Y), \ldots, F_n^*(\lbar Y)$; that is, for every $\lbar a \in \rv^{-1}(1, \ldots, 1)$ and every $i$ we have $\vv (F_i^*(\lbar a)) > 0$. It is actually a simple root because
for every $\lbar a \in \rv^{-1}(1, \ldots, 1)$ we have
\[
\det \partial (F_1^*, \ldots, F_n^*) / \partial \lbar Y (\lbar a) = \biggl(\prod_i e_ic_i \biggr)
\cdot \det \partial (F_1, \ldots, F_n) / \partial \lbar X (\lbar{ac}),
\]
where $\lbar{ac} = (a_1c_1, \ldots, a_nc_n)$, and hence
\[
\vv (\det \partial (F_1^*, \ldots, F_n^*) / \partial \lbar Y (\lbar a)) = \sum_i (-\alpha_i + \gamma_i) +
\sum_i \alpha_i - \sum_i \gamma_i = 0.
\]
Now the lemma follows from the multivariate version of Hensel's
lemma (for example, see~\cite[Corollary 2, p.~224]{bourbaki:1989}).
\end{proof}

\begin{defn}
Let $B, C$ be two $\RV$-pullbacks, $A$ a subset of $B \times C$,
and $U$ a subset of $\rv(B \times C)$. We say that $A$ is a
\emph{$(B, C)$-lift of $U$ from $\RV$ to $\VF$}, or just a
\emph{lift of $U$} for short, if $A \cap (\gp \times \gq)$ is a
bijection from $\gp$ onto $\gq$ for any $\rv$-polydiscs
$\gp \sub B$ and $\gq \sub C$ with $\rv(\gp \times \gq) \in U$. A \emph{partial lift of $U$} is a lift of any subset of $U$.

For any $\RV[k, \cdot]$-isomorphism $F : (U, f) \fun (V, g)$, a lift $F^{\uparrow}$ of $F$ is actually an $(\bb L(U, f), \bb L(V, g))$-lift of the induced function $F_{f,g} : U_f \fun V_g$; that is, $F^{\uparrow}$ is a function on $\bb L(U, f)$ such that each restriction
\[
F^{\uparrow} : \rv^{-1}(f(\lbar u), \lbar u) \fun \rv^{-1}((g \circ F)(\lbar u), F(\lbar u))
\]
is a bijection.
\end{defn}

It would be ideal to lift all definable subsets of $\RV^n \times
\RV^n$ with finite-to-finite correspondence for any substructure
$S$. However, the following crucial lemma fails when $S$ is not
$(\VF, \Gamma)$-generated.

\begin{lem}\label{exists:gamma:polynomial}
Suppose that $S$ is $(\VF, \Gamma)$-generated. Let $\lbar t = (t_1, \ldots, t_{n}) \in \RV$ with $t_n \in \acl(t_1, \ldots, t_{n-1})$. Let $\vrv (\lbar t) = (\gamma_1, \ldots, \gamma_{n}) = \lbar \gamma \in \Gamma$. Then there is a $\lbar \gamma$-polynomial $F(X_1, \ldots, X_{n})$ with coefficients in $\VF(S)$ such that $\lbar t$ is a residue root of $F(\lbar X)$ but is not a residue root of $\partial F(\lbar X) / \partial X_n$.
\end{lem}
\begin{proof}
Write $(t_1, \ldots, t_{n-1})$ as $\lbar t_n$. Let $\phi(\lbar X)$
be a quantifier-free formula such that $\phi(\lbar t_n, X_n)$ defines a finite
subset that contains $t_n$. Without loss of generality we may assume that $\phi(\lbar X)$ is an $\RV$-sort equality such that $\phi(\lbar t_n, X_n)$ defines a finite subset. Since $S$ is $(\VF, \Gamma)$-generated, we may assume that $\phi(\lbar X)$ does not contain parameters from $\RV(S) \mi \rv(\VF(S))$. Hence it is of the form
\[
\lbar X^{\lbar k} \cdot \sum_{\lbar i} (\rv(a_{\lbar i})\cdot
\lbar X^{\lbar i}) = \rv(a) \cdot \lbar X^{\lbar l} \cdot
\sum_{\lbar j} (\rv(a_{\lbar j}) \cdot \lbar X^{\lbar j}),
\]
where $a_{\lbar i}, a, a_{\lbar j} \in \VF(S)$. Fix an $s \in \RV$ such that $\vrv (s \cdot \lbar t^{\lbar k}) = \vrv (s \cdot \rv(a)\cdot \lbar t^{\lbar l}) = 0$. Let $\vrv(s) =
\delta$. Note that $\delta$ is $\lbar t_n$-definable. Let
\[
T_1(\lbar X, s) = \sum_{\lbar i} (s \cdot \rv(a_{\lbar i})\cdot \lbar X^{\lbar i + \lbar k}), \,\,\,\,
T_2(\lbar X, s) = \sum_{\lbar j} (- s \cdot \rv(aa_{\lbar j}) \cdot \lbar X^{\lbar j + \lbar l}).
\]
Consider the $\RV$-sort polynomial $H(\lbar X, s) = T_1(\lbar X, s) + T_2(\lbar X, s)$. For any $r \in \RV$, $H(\lbar t_n, r, s) = 0$ if and only if
\[
\text{either}\quad \sum_{\lbar i} (\rv(a_{\lbar i})\cdot (\lbar t_n, r)^{\lbar i}) = \sum_{\lbar j} (\rv(a_{\lbar j}) \cdot (\lbar t_n, r)^{\lbar j}) = 0 \quad \text{or} \quad \rv(T_1(\lbar t_n, r, s)/s) = \rv(- T_2(\lbar t_n, r, s)/s).
\]
So the equation $H(\lbar t_n, X_n, s) = 0$ defines a finite subset that contains $t_n$ and is actually $\lbar t_n$-definable.

Let $m$ be the maximal exponent of $X_n$ in $H(\lbar X, s)$. For each $i \leq m$ let $H_i(\lbar X, s)$ be the sum of all the monomials $M(\lbar X, s)$
in $H(\lbar X, s)$ such that the exponent of $X_n$ in $M(\lbar X, s)$ is $i$. Replacing $s$ with a variable $Y$ and each $\rv(a)$
with $a$ in $H_i(\lbar X, s)$, we obtain a $\VF$-sort polynomial
$H^*_i(\lbar X, Y)$ for each $i \leq m$. Let
\[
E = \{i \leq m : \vv(H^*_i(\lbar b, c)) = 0 \text{ for all } (\lbar b, c) \in \rv^{-1}(\lbar t, s) \}.
\]
Since $H(\lbar t, s) = 0$, clearly $\abs{E} \neq 1$. Since the equation $H(\lbar t_n, X_n, s) = 0$ defines a finite subset, we actually have $\abs{E} > 1$. Now let
\[
H^*(\lbar X, Y) = \sum_{i \in E} H^*_i(\lbar X, Y) = \sum_{i \in E} Y X_n^i G_i(\lbar X_n) = Y G(\lbar X).
 \]
Since $(\lbar t, s)$ is a residue root of $H^*(\lbar X, Y)$, clearly $G(\lbar X)$
is a $\lbar \gamma$-polynomial with residue value $-\delta$ and $\lbar t$ is a residue root of $G(\lbar X)$. Also, $\lbar t_n$ is not a
residue root of any $G_i(\lbar X_n)$. It follows that, for some $k
< \max E$, $\lbar t$ is a residue root of the $\lbar
\gamma$-polynomial $\partial G(\lbar X) / \partial^k X_n$ but is not a residue root of the $\lbar \gamma$-polynomial $\partial G(\lbar X) / \partial^{k+1} X_n$.
\end{proof}

\begin{rem}\label{def:subseti:in:K}
For definable subsets of the residue field, the situation may be
further simplified. Suppose that $A \sub \K^n$ is definable. Let $\phi(\lbar X)$ be a quantifier-free formula in disjunctive normal form
that defines $A$. It is easily seen by inspection that each conjunct in each disjunct of $\phi(\lbar X)$ is either an $\RV$-sort equality or an $\RV$-sort
disequality, with coefficients in $\K(S)$. So the geometry of
definable subsets in the residue field coincides with its
algebraic geometry. In other words, each definable subset in the
residue field is a constructible subset (in the sense of algebraic
geometry) of a Zariski topological space $\spec \K(S)[\lbar X]$.
\end{rem}

\begin{thm}\label{RV:iso:class:lifted}
Suppose that the substructure $S$ is $(\VF, \Gamma)$-generated.
Let $C \sub (\RV^{\times})^n \times (\RV^{\times})^n$ be a
definable subset such that both $\pr_{\leq n} \rest C$ and
$\pr_{>n} \rest C$ are finite-to-one. Then there is a definable
subset $C^{\uparrow} \sub \VF^n \times \VF^n$ that lifts $C$.
\end{thm}
\begin{proof}
By compactness, the lemma is reduced to showing that for every
$(\lbar t, \lbar s) \in C$ there is a definable lift of some
subset of $C$ that contains $(\lbar t, \lbar s)$. Fix a $(\lbar t,
\lbar s) \in C$ and set $(\lbar \gamma, \lbar \delta) = \vrv(\lbar
t, \lbar s)$. Let $\phi(\lbar X, \lbar Y)$ be a formula that
defines $C$. By Lemma~\ref{exists:gamma:polynomial}, for each $Y_i$ there is a
$(\lbar \gamma, \delta_i)$-polynomial $F_i(\lbar X, Y_i)$ with
coefficients in $\VF(S)$ such that $(\lbar t, s_i)$ is a
residue root of $F_i(\lbar X, Y_i)$ but is not a residue root of
$\partial F_i(\lbar X, Y_i)/ \partial Y_i$. Similarly we obtain
such a $(\gamma_i, \lbar \delta)$-polynomial $G_i(X_i, \lbar Y)$
for each $X_i$. For each $i$, let $a_i (\lbar{X}\lbar{Y})^{\lbar k_i}$
and $b_i (\lbar{X}\lbar{Y})^{\lbar l_i}$ be two monomials with $a_i, b_i
\in \VF(S)$ such that
\[
F_i^*(\lbar X, \lbar Y) + G_i^*(\lbar X, \lbar Y) = a_i (\lbar{X}\lbar{Y})^{\lbar k_i} F_i(\lbar X, Y_i) + b_i (\lbar{X}\lbar{Y})^{\lbar l_i}G_i(X_i, \lbar Y)
\]
is a $(\lbar \gamma, \lbar \delta)$-polynomial. Let $\alpha_i$ be
the residue value of $F_i^*(\lbar X, \lbar Y) + G_i^*(\lbar X,
\lbar Y)$. Note that for any $(\lbar a, \lbar b) \in
\rv^{-1}(\lbar t, \lbar s)$ we have
\[
\vv (\partial F_i^* / \partial Y_i(\lbar a, \lbar b)) = \vv(a_i(\lbar {ab})^{\lbar k_i}) + \vv(\partial F_i / \partial Y_i(\lbar a, \lbar b)) = \alpha_i - \delta_i
\]
and for $j \neq i$ we have
\[
\vv (\partial F_i^* / \partial Y_j(\lbar a, \lbar b)) = \vv(a_i) + \vv(\partial (\lbar{X}\lbar{Y})^{\lbar k_i} / \partial Y_j (\lbar a, \lbar b)) + \vv (F_i(\lbar a, b_i)) > \alpha_i - \delta_j.
\]
Therefore,
\[
\vv (\det \partial (F^*_1, \ldots, F^*_n) / \partial \lbar Y (\lbar a, \lbar b))
= \vv \bigl( \prod_i \partial F^*_i /\partial Y_i (\lbar a, \lbar b) \bigr)
= \sum_i \alpha_i - \sum_i \delta_i.
\]
This shows that $\lbar s$ is a simple common residue root of
$F^*_1(\lbar a, \lbar Y), \ldots, F^*_n(\lbar a, \lbar Y)$ for
any $\lbar a \in \rv^{-1}(\lbar t)$. Similarly $\lbar t$ is a
simple common residue root of $G^*_1(\lbar X, \lbar b), \ldots,
G^*_n(\lbar X, \lbar b)$ for any $\lbar b \in \rv^{-1}(\lbar
s)$.

Now for each $i$ we choose a pair of integers $p_i, q_i$. Consider the $(\lbar \gamma, \lbar \delta)$-polynomials
\[
H_i(\lbar X, \lbar Y) = p_i F_i^*(\lbar X, \lbar Y) + q_i G_i^*(\lbar X, \lbar Y).
\]
Let $\sigma \in S_n$ be a permutation and $\tau(\lbar X, \lbar Y)$ a term in the expansion of the product $\prod_i \partial H_i(\lbar X, \lbar Y) / \partial Y_{\sigma(i)}$. The coefficient $c_{\tau}$ of $\tau(\lbar X, \lbar Y)$ is of the form $\prod_i m_i$, where $m_i $ is either $p_i$ or $q_i$. Suppose that
\[
\vv(\tau (\lbar a, \lbar b)) = \sum_i \alpha_i - \sum_i \delta_i
\]
for some (hence all) $(\lbar a, \lbar b) \in \rv^{-1}(\lbar t, \lbar s)$. Then $\rv(\tau(\lbar X, \lbar Y))$ is constant on $\rv^{-1}(\lbar t, \lbar s)$, which is denoted by $\rv(\tau)$. Observe that there is only one such term with coefficient $\prod_i p_i$, namely $\prod_i \partial (p_i F^*_i) /\partial Y_i$. Let $\tau_i$ enumerate all such terms other than $\prod_i \partial (p_i F^*_i) /\partial Y_i$. It is not hard to see that $p_i, q_i$ may be chosen so that
\[
1 + \sum_i \rv(\tau_i) / \rv \bigl( \prod_i \partial (p_i F^*_i) / \partial Y_i \bigr) \neq 0.
\]
This implies that, for all $(\lbar a, \lbar b) \in \rv^{-1}(\lbar t, \lbar s)$,
\[
\vv (\det \partial (H_1, \ldots, H_n) / \partial \lbar Y (\lbar a, \lbar b))
= \sum_i \alpha_i - \sum_i \delta_i
\]
and hence $\lbar s$ is a simple common residue root of the $\lbar \delta$-polynomials $H_1(\lbar a, \lbar Y), \ldots, H_n(\lbar a, \lbar Y)$ for any $\lbar a \in \rv^{-1}(\lbar t)$. In fact the choice of $p_i, q_i$ can be improved so that we also have, for all $(\lbar a, \lbar b) \in \rv^{-1}(\lbar t, \lbar s)$,
\[
\vv (\det \partial (H_1, \ldots, H_n) / \partial \lbar X (\lbar a, \lbar b))
= \sum_i \alpha_i - \sum_i \gamma_i
\]
and hence $\lbar t$ is a simple common residue root of the $\lbar \gamma$-polynomials $H_1(\lbar X, \lbar b), \ldots, H_n(\lbar X, \lbar b)$ for any $\lbar b \in \rv^{-1}(\lbar s)$.
By Lemma~\ref{hensel:lemma}, for each
$\lbar a \in \rv^{-1}(\lbar t)$ there is a unique $\lbar b \in
\rv^{-1}(\lbar s)$ such that $\bigwedge_i H_i(\lbar a, \lbar b)
= 0$, and vice versa.
\end{proof}

\begin{cor}\label{L:semigroup:hom}
Suppose that the substructure $S$ is $(\VF, \Gamma)$-generated.
The map $\mathbb{L}$ induces surjective homomorphisms between various
Grothendieck semigroups, for example:
\[
\gsk \RV[k, \cdot] \fun \gsk \VF[k, \cdot],\quad
\gsk \RV[k] \fun \gsk \VF[k].
\]
\end{cor}
\begin{proof}
For any $\RV[k, \cdot]$-isomorphism $F : (U, f) \fun (V, g)$ and
any $\lbar u \in U$, by definition, $\wgt (f(\lbar u)) = \wgt ((g
\circ F)(\lbar u))$. Let
\[
C = \set{(f(\lbar u), (g \circ F)(\lbar u)) : \lbar u \in U} \sub \RV^k \times \RV^k.
\]
By Theorem~\ref{RV:iso:class:lifted} there is a lift $C^{\uparrow}$ of $C$, which induces a $\VF[k, \cdot]$-isomorphism between $\bb L(U, f)$ and $\bb L(V, g)$. So $\bb L$ induces a map on the isomorphism classes, which is clearly a semigroup homomorphism. By Corollary~\ref{L:surjective} it is surjective. The other cases are handled similarly.
\end{proof}

\section{More on structural properties}\label{section:more:stru}

\begin{lem}\label{points:apart:fin}
Let $A \sub \VF^n$ be a definable subset. Suppose that there is a $\gamma \in \Gamma$ such that $\go(\lbar a', \gamma) \cap \go(\lbar a'', \gamma) = \0$ for every $\lbar a'$, $\lbar a'' \in A$. Then $A$ is finite.
\end{lem}
\begin{proof}
We do induction on $n$. The base case $n = 1$ just follows from $C$-minimality. For the inductive step, consider the subset $\pr_1 (A) = A_1$. If $A_1$ is finite then by the inductive hypothesis $\fib(A, a)$ is finite for every $a \in A_1$ and hence $A$ is finite. If $A_1$ is infinite then by $C$-minimality there is an open ball $\gb \sub A_1$ with $\rad(\gb) > \gamma$. For any $a'\in \gb$, $a'' \in \gb$, $\lbar b' \in \fib(A, a')$, and $\lbar b'' \in \fib(A, a'')$, if $\go(\lbar b', \gamma) \cap \go(\lbar b'', \gamma) \neq \0$ then $\go((a', \lbar b'), \gamma) \cap \go((a'', \lbar b''), \gamma) \neq \0$, contradicting the assumption. Therefore, by the inductive hypothesis again, $\bigcup_{a \in \gb} \fib(A, a)$ is finite. So there is a $\lbar b \in \bigcup_{a \in \gb} \fib(A, a)$ such that $\fib(A, \lbar b) \cap \gb$ is infinite, contradiction again.
\end{proof}

\begin{lem}\label{iso:point:dim:less}
Let $f :\VF^n \fun \VF^m$ be a definable function. Let $A \sub \VF^n$ be the definable subset of those $\lbar a \in \VF^n$ such that there are $\epsilon$, $\delta \in \Gamma$ with
\[
\go(\lbar a, \delta)   \cap  f^{-1}(\go(f(\lbar a), \epsilon)) = \set{\lbar a}.
\]
Then $\dim_{\VF} (A) < n$.
\end{lem}
\begin{proof}
For each $\lbar a \in A$ let $(\epsilon_{\lbar a}, \delta_{\lbar a}) \in \Gamma^2$ be an $\lbar a$-definable pair that satisfies the condition above, which exists by $o$-minimality. Let $h : A \fun \Gamma^2$ be the definable function given by $\lbar a \efun (\epsilon_{\lbar a}, \delta_{\lbar a})$. Suppose for contradiction that $\dim_{\VF}(A) = n$. Then, by compactness and Lemma~\ref{full:dim:open:poly}, there is a pair $(\epsilon_{\lbar a}, \delta_{\lbar a}) \in \Gamma^2$ such that $h^{-1}(\epsilon_{\lbar a}, \delta_{\lbar a})$ contains an open polydisc $\gp$. Without loss of generality we may assume $\lbar a \in \gp$. Fix an $\lbar a$-definable $\gamma \geq \delta_{\lbar a}$. If $\lbar a'$, $\lbar a'' \in \go(\lbar a, \gamma)$ are distinct then $\go(f(\lbar a'), \epsilon_{\lbar a}) \cap \go(f(\lbar a''), \epsilon_{\lbar a}) = \0$. By Lemma~\ref{points:apart:fin}, $f(\go(\lbar a, \gamma))$ is finite, which is a contradiction.
\end{proof}

Let $A$ be a definable subset with $\dim_{\VF}(A) = n$. A property holds \emph{almost everywhere on $A$} or \emph{for almost every element in $A$} if there is a definable subset $B \sub A$ with $\dim_{\VF} (B) < n$ such that the property holds with respect to $A \mi B$. For example, if $f : \VF^n \fun \VF^m$ is a definable function, then the property that defines the subset $A$ in Lemma~\ref{iso:point:dim:less} \emph{does not} hold almost everywhere on $\VF^n$. This terminology is also used with respect to $\RV$-dimension.

\begin{lem}\label{fun:dim:1:val:cons}
Let $f :\VF \times \VF^k \fun \VF^m$ be a definable function. Then there are a definable subset $A \sub \VF \times \VF^k$ over $\VF^k$ and a finite set $E$ of positive rational numbers such that
\begin{enumerate}
  \item $\VF \mi \fib(A, \lbar b)$ is finite for all $\lbar b \in \VF^k$,
  \item for every $\lbar a = (a, \lbar b) \in A$ there are $\lbar a$-definable $\epsilon, \delta \in \Gamma$ and a number $k \in E$ such that either $f \rest \go(a, \delta) \times \{\lbar b\}$ is constant or, for any $a' \in \go(a, \delta)$,
\[
\vv(f(a', \lbar b) - f(a, \lbar b)) = \epsilon + k \vv(a' - a).
\]
\end{enumerate}
\end{lem}
\begin{proof}
For every $\lbar b \in \VF^k$ let $B_{\lbar b} \sub \VF \times \{\lbar b\}$ be as given by Lemma~\ref{iso:point:dim:less} with respect to the function $f \rest \VF \times \{\lbar b\}$. By compactness $A = \VF^{k +1} \mi \bigcup_{\lbar b \in \VF^k} B_{\lbar b}$ is definable. Let $\phi(X_1, X_2, \lbar Y, Z)$ be a quantifier-free $\lan{v}$-formula, possibly with additional parameters from $\VF$, that defines the function on $\VF^2 \times \VF^k$ given by
\[
(a', a, \lbar b)  \efun \vv(f(a', \lbar b) - f(a, \lbar b)).
\]
Fix an $\lbar a = (a, \lbar b) \in A$ such that $f \rest \VF \times \{\lbar b\}$ is not constant on any open ball around $a$. For any term of the form $\vv (G(X_1, X_2, \lbar Y))$ in $\phi(X_1, X_2, \lbar Y, Z)$ there is an $\lbar a$-definable $\alpha \in \Gamma \cup \{\infty\}$ and an integer $l \geq 0$ such that, for any $a + d \in \VF$, if $\vv(d)$ is sufficiently large then
\[
\vv(G(a + d, a, \lbar b)) = \alpha + l\vv(d).
\]
Therefore, there is an $\epsilon \in \Gamma \cup \{\infty\}$ and a rational number $k \geq 0$ such that for any sufficiently large $\delta \in \Gamma$, the formula
\[
\vv(X) > \delta \wedge \phi(a+X, a, \lbar b, Z)
\]
defines a function on $\go(a, \delta) \times \{\lbar b\}$ that is given by the equation $Z = \epsilon + k \vv(X)$. Note that, by the choice of $\lbar a$, we actually must have $k > 0$ and $\epsilon \neq \infty$. Since $\Gamma$ is $o$-minimal, $\epsilon$ and some $\delta$ are $\lbar a$-definable. Now it is easy to see that the number $k$ is provided by the exponents of $X_1$ in $\phi(X_1, X_2, \lbar Y, Z)$ and hence there are only finitely many choices.
\end{proof}

\begin{lem}\label{rad:equal:1}
Let $\ga, \gb$ be open balls around $0$ and $f : \ga \fun \gb$ a definable bijection that takes open balls around $0$ to open balls around $0$. Then there are definable $\gamma, \epsilon \in \Gamma$ such that $\vv(f(a)) = \epsilon + \vv(a)$ for every $a \in \go(0, \gamma)$.
\end{lem}
\begin{proof}
By the proof of Lemma~\ref{fun:dim:1:val:cons} we may assume that there is a definable $\epsilon \in \Gamma$ and a positive rational number $k$ such that $\vv(f(a)) = \epsilon + k \vv(a)$. We need to show that $k=1$.

Suppose for contradiction $k \neq 1$. Let $\phi(X, Y)$ be a quantifier-free $\lan{v}$-formula, possibly with additional parameters from $\VF$, that defines $f$. Let $F_i(X, Y)$ be the occurring polynomials of $\phi(X, Y)$. If $a \in \ga$ then $F_i(a, f(a)) = 0$ for some $i$, since otherwise $f^{-1}(f(a))$ would be infinite. By $C$-minimality, we may shrink $\ga$ if necessary so that, for every $a \in \ga$, $F_i(a, f(a)) = 0$ if and only if $i \leq m$. For every $F_j(X, Y)$ with $j > m$, since $k \neq 1$, we may shrink $\ga$ again so that, for some monomial $c X^{l} Y^{n}$, for every $a \in \ga$, and for every $r, s \in\VF$ with $\vv(r) = \vv(a)$ and $\vv(s) = \vv(f(a))$, we have
\[
\vv(F_j(r, s)) = \vv(c r^{l} s^{n}) = \vv(c a^{l} f(a)^{n}) = \vv(F_j(a, f(a))).
\]
Now, using the division algorithm, there are rational functions $G(X, Y) \in \VF(X)[Y]$ and $H(X, Y) \in \VF(Y)[X]$ such that, possibly after shrinking $\ga$ again,
\begin{enumerate}
  \item every solution of $G(a, Y) = 0$ is a solution of $\bigwedge_{i \leq m} F_i(a, Y) = 0$ for every $a \in \ga$,
  \item every solution of $H(X, b) = 0$ is a solution of $\bigwedge_{i \leq m} F_i(X, b) = 0$ for every $b \in \gb$.
  \item taking derivatives and using the division algorithm again if necessary, for every $a \in\ga$, $f(a)$ is not a repeated root of $G(a, Y)$ and $a$ is not a repeated root of $H(X, f(a))$,
\end{enumerate}
Moreover, we may assume that, if we write $G(X, Y)$ as $\sum_i G_i(X) Y^i$ then there are indices $i < i'$ such that for every $a \in\ga$
\[
  \vv(f(a))^{i' - i} = \vv(G_{i'}(a)/\vv(G_i(a))) > 0;
\]
Similarly for $H(X, Y)$. Observe that if $i' - i > 1$ then for every $a \in \ga$ there is a root $r \neq f(a)$ of $G(a, Y)$ such that $\vv(F_j(a, f(a))) = \vv(F_j(a, r))$ for all $j > m$ and hence $\phi(a, r)$ holds, which is a contradiction. So $i' = i + 1$. Since the radius of $\ga$ is sufficiently large, we conclude that $k$ must be a positive integer. Symmetrically $1/k$ is also a positive integer and hence $k=1$, contradicting the assumption $k \neq 1$.
\end{proof}

\begin{lem}\label{ball:shrink:RV:cons}
Let $A, B \sub \VF$ be infinite subsets and $f : A \fun B$ a definable bijection. Then for almost all $a \in A$ there are $a$-definable $\delta \in \Gamma$ and $t \in \RV^{\times}$ such that, for any $b, b' \in \go(a, \delta)$,
\[
\rv(f(b) - f(b')) = t \rv(b - b').
\]
\end{lem}
\begin{proof}
Let $A' \sub A$ be a definable subset such that $A \mi A'$ is finite and for every $a \in A'$ there are $\epsilon_a, \delta_a \in \Gamma$ given as in Lemma~\ref{fun:dim:1:val:cons}. Translating $A, B$ to $A- a, B -f(a)$ and applying Lemma~\ref{rad:equal:1}, we see that $\delta_a$ may be chosen so that
\[
\vv(f(b) - f(a)) = \epsilon_a + \vv(b - a)
\]
for any $b \in \go(a, \delta_a)$. Let $D_a = (\go(a, \delta_a) - a) \mi \set{0}$ and $g_a : D_a \fun \RV$ the function given by
\[
d \efun \rv(f(d +a) - f(a))/\rv(d).
\]
Since $\vrv(g_a(D_a))$ is bounded from both above and below, by Lemma~\ref{fun:bounded:cons}, there is a $\beta_a \in \Gamma$ such that $g_a (\go(0, \beta_a) \mi \set{0}) = t_a$. Let $h : A' \fun \Gamma \times \RV$ be the function given by $a \efun (\delta_a, t_a)$. By compactness and Corollary~\ref{function:rv:to:vf:finite:image}, there are only finitely many $a \in A'$ that is isolated in $h^{-1}(\delta_a, t_a)$. On the other hand, if $\go(a, \gamma) \sub h^{-1}(\delta, t)$ with $\gamma \geq \delta$ then clearly for any $b, b' \in \go(a, \gamma)$,
\[
\rv(f(b) - f(b')) = t \rv(b - b'),
\]
as required.
\end{proof}

Lemma~\ref{fun:dim:1:val:cons} can be generalized to multivariate functions, but only with inequality:

\begin{lem}\label{VF:fun:reg:away}
Let $f :\VF^n \times \VF^k \fun \VF^m$ be a definable function. Then there are a definable subset $A \sub \VF^n \times \VF^k$ over $\VF^k$ and a positive rational number $k$ such that
\begin{enumerate}
  \item $\dim_{\VF}(\VF^n \mi \fib(A, \lbar b)) < n$ for all $\lbar b \in \VF^k$,
  \item for every $\lbar x = (\lbar a, \lbar b) \in A$ there are $\lbar x$-definable $\epsilon, \delta \in \Gamma$ such that for any $\lbar a' \in \go(\lbar a, \delta)$,
\[
\vv(f(\lbar a', \lbar b) - f(\lbar a, \lbar b)) \geq \epsilon + k \vv(\lbar a' - \lbar a).
\]
\end{enumerate}
\end{lem}
\begin{proof}
We do induction on $n$. The base case $n = 1$ is readily implied by Lemma~\ref{fun:dim:1:val:cons}.

We proceed to the inductive step. By the inductive hypothesis, there are a definable subset $A_1 \sub \VF^{n-1} \times \VF^{k+1}$ over $\VF^{k+1}$ and a positive rational number $k_1$ with respect to which the conclusion of the lemma holds. Similarly, there are a definable subset $A_2 \sub \VF^{n-1} \times \VF \times \VF^{k}$ over $\VF^{k+ n -1}$ and a positive rational number $k_2$ with respect to which the conclusion of the lemma holds.

Let $k = \min\{k_1, k_2\}$. Fix a $\lbar c \in \VF^k$. We shall concentrate on the subsets $\fib(A_1, \lbar c), \fib(A_2, \lbar c)$, which, for simplicity, are respectively written as $C_1, C_2$. Also we shall suppress mentioning $\lbar c$ as parameters. Set $C = C_1 \cap C_2$. Note that, by compactness, $\dim_{\VF}(\VF^n \mi C) < n$. Consider any $(\lbar a, b) \in C_1$. Let $(\epsilon_b, \delta_b) \in \Gamma^2$ be an $(\lbar a, b)$-definable pair such that, for any $\lbar a' \in \go(\lbar a, \delta_b)$,
\[
\vv(f(\lbar a', b) - f(\lbar a, b)) \geq \epsilon_b + k \vv(\lbar a' - \lbar a).
\]
Let $h_{\lbar a} : \fib(C_1, \lbar a) \fun \Gamma^2$ be the $\lbar a$-definable function given by $(\lbar a, b) \efun (\epsilon_b, \delta_{b})$. For each $(\epsilon, \delta) \in \Gamma^2$ let $B_{\epsilon, \delta}$ be the topological interior of $h_{\lbar a}^{-1}(\epsilon, \delta)$. Let
\[
B_{\lbar a} = \bigcup_{(\epsilon, \delta) \in \Gamma^2} B_{\epsilon, \delta} \quad \text{and} \quad B = \bigcup_{\lbar a \in \pr_{< n} (C_1)} (\set{\lbar a} \times (\fib(C_1, \lbar a) \mi B_{\lbar a})).
\]
By \cmin-minimality, $\dim_{\VF} (h_{\lbar a}^{-1}(\epsilon, \delta) \mi B_{\epsilon, \delta}) = 0$ for every $(\epsilon, \delta) \in \Gamma^2$ and hence, by Lemma~\ref{dim:VF:pullback} and Lemma~\ref{dim:RV:fibers}, $\dim_{\VF} (\fib(C_1, \lbar a) \mi B_{\lbar a}) = 0$ and $\dim_{\VF} (B) < n$.

Let $(\lbar a_1, b_1) \in C \mi B$ and $h_{\lbar a_1} (b_1) =(\epsilon_1, \delta_1)$. Since the corresponding interior $B_{\epsilon_1, \delta_1}$ is nonempty, there are $(\lbar a_1, b_1)$-definable $\delta_2$, $\epsilon_2 \in \Gamma$ such that $\go(b_1, \delta_2) \sub B_{\epsilon_1, \delta_1}$ and, for any $b_2 \in \go(b_1, \delta_2)$,
\[
\vv(f(\lbar a_1, b_2) - f(\lbar a_1, b_1)) \geq \epsilon_2 + k \vv(b_2 - b_1).
\]
On the other hand, for any $b_2 \in \go(b_1, \delta_2)$ and any $\lbar a_2 \in \go(\lbar a_1, \delta_1)$,
\[
\vv(f(\lbar a_2, b_2) - f(\lbar a_1, b_2)) \geq \epsilon_1 + k \vv(\lbar a_2 - \lbar a_1).
\]
We then have
\begin{align*}
      & \vv(f(\lbar a_2, b_2) - f(\lbar a_1, b_1))\\
 \geq & \min\set{\vv(f(\lbar a_1, b_2) - f(\lbar a_1, b_1)), \vv(f(\lbar a_2, b_2) - f(\lbar a_1, b_2))}\\
 \geq & \min\set{\epsilon_1, \epsilon_2} + \min\set{k \vv(b_2 - b_1), k \vv(\lbar a_2 - \lbar a_1)}\\
 = & \min\set{\epsilon_1, \epsilon_2} + k \vv((\lbar a_2, b_2) - (\lbar a_1, b_1)).
\end{align*}
Now the lemma follows from compactness.
\end{proof}

Clearly this lemma holds with respect to any definable function $f : A \fun \VF^m$ with $A \sub \VF^n$ and $\dim_{\VF}(A) = n$, since $f$ may be extended to $\VF^n$ by sending $\VF^n \mi A$ to any definable tuple in $\VF^m$. In application we usually take $k = 0$.

\begin{lem}\label{fun:almost:cont}
Let $f : \VF^n \fun \VF^m$ be a definable function. Then there is a definable closed subset $A \sub \VF^n$ with $\dim_{\VF}(A) < n$ such that $f \rest (\VF^n \mi A)$ is continuous with respect to the valuation topology.
\end{lem}
\begin{proof}
Let $A \sub \VF^n$ be the definable subset of ``discontinuous points'' of $f$; that is, $\lbar a \in A$ if and only if there is a $\gamma \in \Gamma$ such that $f^{-1}(\go(f(\lbar a), \gamma))$ fails to contain any open polydisc around $\lbar a$. Let $\lbar A$ be the topological closure of $A$, which is definable, and set $f_1 = f \rest (\VF^n \mi \lbar A)$. For any $\lbar a \in \VF^n \mi \lbar A$ and any $\gamma \in \Gamma$, since $f^{-1}(\go(f(\lbar a), \gamma))$ contains an open polydisc around $\lbar a$, $f_1^{-1}(\go(f(\lbar a), \gamma))$ must also contain an open polydisc around $\lbar a$. So it is enough to show that $\dim_{\VF} (\lbar A) < n$, which, by Lemma~\ref{full:dim:open:poly}, is equivalent to showing that $\dim_{\VF} (A) < n$.

Suppose for contradiction that $\dim_{\VF} (A) = n$. Let $A' \sub A$ be the definable subset given by Lemma~\ref{VF:fun:reg:away} with respect to $f$. Since $\dim_{\VF} (A') = n$, by Lemma~\ref{full:dim:open:poly} again, $A'$ contains an open polydisc $\gp$. Fix an $\lbar a \in \gp$ and let $\gamma \in \Gamma$ be such that $f^{-1}(\go(f(\lbar a), \gamma))$ fails to contain any open ball around $\lbar a$. By Lemma~\ref{VF:fun:reg:away}, there are $\epsilon$, $\delta \in \Gamma$ such that
\begin{enumerate}
 \item $\go(\lbar a, \delta) \sub \gp$,
 \item $\epsilon + \delta > \gamma$,
 \item for any $\lbar b \in \go(\lbar a, \delta)$ with $\lbar b \neq \lbar a$, $\vv(f(\lbar b) - f(\lbar a)) \geq \epsilon + \delta$.
\end{enumerate}
So $\go(\lbar a, \delta) \sub f^{-1}(\go(f(\lbar a), \gamma))$, contradiction.
\end{proof}

\begin{defn}
A function $f : \VF^n \fun \mdl P(\RV^m)$ is \emph{locally constant at $\lbar a$} if there is an open subset $U_{\lbar a} \sub \VF^n$ containing $\lbar a$ such that $f \rest U_{\lbar a}$ is constant. If $f$ is locally constant at every point in an open subset $A$ then $f$ is locally constant on $A$.
\end{defn}

\begin{lem}\label{fun:almost:loc:con}
Let $f : \VF^n \fun \mdl P(\RV^m)$ be a definable function. Then $f$ is locally constant almost everywhere.
\end{lem}
\begin{proof}
We do induction on $n$. For the base case $n = 1$, let $A \sub \VF$ be the definable subset of those $a \in \VF$ such that $f$ is not constant on any $\go(a, \gamma)$. Let $\lbar A$ be the topological closure of $A$. It is enough to show that $\dim_{\VF} (\lbar A) = 0$, which, by $C$-minimality, is equivalent to showing that $A$ is finite. Suppose for contradiction that $A$ is infinite. By $C$-minimality again there is a definable $\gamma \in \Gamma$ such that $A$ contains infinitely many cosets of $\go(0, \gamma)$. By Lemma~\ref{fun:quo:rep}, $f$ fails to be constant on only finitely many cosets of $\go(0, \gamma)$, contradiction.

We proceed to the inductive step. For any $\lbar a = (a_1, \lbar a_1) \in \VF^n$, let $(\alpha_{\lbar a}, \beta_{\lbar a}) \in \Gamma^2$ be an $\lbar a$-definable pair such that $f$ is constant on both $\go(a_1, \alpha_{\lbar a}) \times \set{\lbar a_1}$ and $\set{a_1} \times \go(\lbar a_1, \beta_{\lbar a})$. If no such pair exists then set $\alpha_{\lbar a} = \beta_{\lbar a} = \infty$. Let $g : \VF^n \fun \Gamma^2$ be the function given by $\lbar a \efun (\alpha_{\lbar a}, \beta_{\lbar a})$. By the inductive hypothesis and compactness, $\dim_{\VF} (g^{-1}(\infty, \infty)) < n$. For each $(\alpha, \beta) \in \Gamma^2$ let $B_{\alpha, \beta}$ be the topological interior of $g^{-1}(\alpha, \beta)$. By Lemma~\ref{full:dim:open:poly},
\[
\dim_{\VF} (g^{-1}(\alpha, \beta) \mi B_{\alpha, \beta}) < n.
\]
Let $B = \bigcup_{(\alpha, \beta) \in \Gamma^2} B_{\alpha, \beta}$. By compactness, $\dim_{\VF} (\VF^n \mi B) < n$. For any $\lbar a = (a_1, \lbar a_1) \in B$, since $B_{\alpha_{\lbar a}, \beta_{\lbar a}}$ contains an open polydisc around $\lbar a$, clearly for any sufficiently large $\gamma$ and any $(a'_1, \lbar a'_1) \in \go(\lbar a, \gamma)$ we have $f(a_1, \lbar a_1) = f(a_1, \lbar a'_1) = f(a'_1, \lbar a'_1)$. So $f$ is locally constant on $B$.
\end{proof}

\section{Differentiation}\label{section:diff}

We shall extend the results in Section~\ref{section:RV:product} and Section~\ref{section:lifting} to finer categories of definable subsets with volume forms. To define these categories we first need a notion of the Jacobian in the $\VF$-sort. There are two approaches, which essentially give the same data. The first one is an analogue of the classical analytic approach, where we first define differentiation and the notion of ``approaching a point'' is expressed via valuation. This method makes certain computations very easy (see Lemma~\ref{special:tran:vol:pre} and Lemma~\ref{jcb:chain}). The second approach is an algebraic one, where we are reduced to the case of computing the Jacobian of a regular map between varieties over $\VF$. The Jacobian in the $\RV$-sort will also be defined in this way. This makes the compatibility of the Jacobian in both sorts transparent.

In the discussion below it is convenient to think that there is a ``point at infinity'' in the $\VF$-sort, denoted by $p_{\infty}$. The set $\VF \cup \set{p_{\infty}}$ is denoted by $\PP(\VF)$. Balls around $p_{\infty}$ are defined in a reversed way. For example, for any $\gamma \in \Gamma$, the open ball $\go(p_{\infty}, \gamma)$ around $p_{\infty}$ of radius $\gamma$ is the subset $\VF \mi \gc(0, - \gamma)$. Note the negative sign in front of $\gamma$. We emphasize that $p_{\infty}$ will not be treated as a real point. It is merely a notational device that allows us to discuss complements of balls around $0$ more efficiently.

\begin{defn}
Let $A \sub \VF^n$, $f : A \fun \mdl P(\VF^m)$ a definable function, $\lbar a \in \VF^n$, and $L \sub \PP(\VF)^m$. We say that $L$ is a \emph{limit set of $f$ at $\lbar a$}, written as $\lim_{A \rightarrow \lbar a} f \sub L$, if for every $\epsilon \in \Gamma$ there is a $\delta \in \Gamma$ such that if $\lbar c \in \go(\lbar a, \delta) \cap (A \mi \lbar a)$ then $f(\lbar c) \sub \bigcup_{\lbar b \in L'} \go(\lbar b, \epsilon)$ for some $L' \sub L$.
\end{defn}

A limit set $L$ of $f$ at $\lbar a$ is \emph{minimal} if no proper subset of $L$ is a limit set of $f$ at $\lbar a$. Observe that if $\lim_{A \rightarrow \lbar a} f \sub L$ and $\lbar b \in L$ is not isolated in $L$ then actually $\lim_{ A \rightarrow \lbar a} f \sub L \mi \{\lbar b \}$. So in a minimal limit set every element is isolated. Moreover, if a minimal limit set $L$ exists then its topological closure $\lbar L$ is unique:

\begin{lem}
Let $L_1$, $L_2 \sub \VF^m$ be two minimal limit sets of $f$ at $\lbar a$ and $\lbar L_1$, $\lbar L_2$ their topological closures. Then $\lbar L_1 = \lbar L_2$.
\end{lem}
\begin{proof}
Suppose for contradiction that, say, $\lbar L_1 \mi \lbar L_2 \neq \0$ and hence there is a $\lbar b \in L_1 \mi \lbar L_2$. So there is an $\epsilon \in \Gamma$ such that $\go(\lbar b, \epsilon) \cap L_2 = \0$. Let $\delta \in \Gamma$ be such that, for all $\lbar c \in \go(\lbar a, \delta) \cap ( A \mi \lbar a)$, $f(\lbar c) \sub \bigcup_{\lbar d \in L'_2} \go(\lbar d, \epsilon)$ for some $L'_2 \sub L_2$. Since $\go(\lbar b, \epsilon) \cap \go(\lbar d, \epsilon) = \0$ for any $\lbar d \in L_2$, we see that $L_1 \mi \{\lbar b \}$ is a limit set of $f$ at $a$, contradicting the minimality condition on $L_1$. So $\lbar L_1 \sub \lbar L_2$ and symmetrically $\lbar L_2 \sub \lbar L_1$.
\end{proof}

This lemma justifies the equality $\lim_{ A \rightarrow \lbar a} f = L$ when $L$ is a closed (hence the unique) minimal limit set of $f$ at $\lbar a$.

\begin{lem}\label{union:lim}
Let $f_1$, $f_2 : A \fun \mdl P( \VF^m)$ be definable functions with $\lim_{A \rightarrow \lbar a} f_i = L_i$, then $\lim_{A \rightarrow \lbar a} (f_1 \cup f_2) = L_1 \cup L_2$.
\end{lem}
\begin{proof}
Let $f = f_1 \cup f_2$ and $L = L_1 \cup L_2$. Clearly $L$ is a closed limit set of $f$ at $\lbar a$. We need to show that it is minimal. To that end, fix a $\lbar b \in L_1$. If $\lbar b \in L_1 \cap L_2$ then, since $\lbar b$ is isolated in both $L_1$ and $L_2$, there is an $\epsilon \in \Gamma$ such that $\go(\lbar b, \epsilon) \cap (L \mi \{\lbar b \}) = \0$. If $\lbar b \in L_1 \mi L_2$ then, since $L_2$ is closed, there is again an $\epsilon \in \Gamma$ such that $\go(\lbar b, \epsilon) \cap (L \mi \{\lbar b\}) = \0$. Now, since $L_1$ is a limit set of $f_1$ at $\lbar a$ but $L_1 \mi \{\lbar b\}$ is not, we see that $L \mi \{\lbar b\}$ cannot be a limit set of $f$ at $\lbar a$. This shows that $L$ is minimal.
\end{proof}

\begin{lem}\label{fin:lim:bounded}
Let $f : A \fun \mdl P(\VF^m)$ be a definable function with finite images. Let $k$ be the maximal size of $f(c)$. Let $a \in \VF$ and suppose that there is an open ball $\gb$ containing $a$ such that $\gb \mi \set{a} \sub A$. Suppose that $\lim_{A \rightarrow a} f = L$. If $L$ is finite then $\abs{L} \leq k$.
\end{lem}
\begin{proof}
Let $L = \{\lbar b_1, \ldots, \lbar b_l \}$ and suppose for contradiction that $l > k$. Let $\alpha \in \Gamma$ be such that $\go(\lbar b_i, \alpha) \cap \go(\lbar b_j, \alpha) = \0$ whenever $i \neq j$. Without loss of generality we may assume that $A \mi \set{a} =  \gb \mi \set{a}$ and $\bigcup f(A) \sub \bigcup_i \go(\lbar b_i, \alpha)$. For each $D \sub L$ with $\abs{D} = k$ let
\[
A_D = \bigl \{c \in A : f(c) \sub \bigcup_{\lbar b_i \in D} \go(\lbar b_i, \alpha) \bigr\}.
\]
Each $A_D$ is $\seq{L, \alpha}$-definable. By $C$-minimality, some $A_D \cup \set{a}$ contains an open ball around $a$ and hence $\lim_{A \rightarrow a} f = \lim_{A_D \rightarrow a} (f \rest A_D) \sub D$, contradicting the assumption that $\lim_{A \rightarrow a} f = L$.
\end{proof}

Here is the key lemma that makes the definition of differentiation in $\VF$ below work. It is essentially a variation on a fundamental property of henselian fields, see~\cite[Proposition, p.~70]{koblitz:padic:analysis}.

\begin{lem}\label{lim:exists}
Let $\gb \sub \VF$ be a ball containing $0$ and $A \sub (\gb \mi \set{0}) \times \VF^m$ a definable function $\gb \mi \set{0} \fun \mdl P (\VF^m)$ with finite images. Then there is a definable finite subset $L \sub \PP(\VF)^m$ such that $\lim_{\gb \mi \set{0} \rightarrow 0} A = L$.
\end{lem}
\begin{proof}
Without loss of generality we may assume that $\gb$ is an open ball. Let $\ga = \gb \mi \set{0}$. We first consider the basic case: $m = 1$ and there is a polynomial $G(X, Y) \in \VF(S)[X, Y]$ such that $(a, b) \in A$ if and only if $a \in  \ga$ and $G(a,b) = 0$.

Fix an $\epsilon \in \Gamma$. Write $G(X, Y)$ as $Y^m H(X) G^*(X, Y)$, where $H(X) \in \VF(S)[X]$ and $G^*(X, Y) \in \VF(S)[X, Y]$ is of the form
\[
H_n(X)Y^n + \cdots + H_0(X),
\]
where the polynomials $H_j(X)Y^j \in \VF(S)[X, Y]$ are relatively prime. Shrinking $\ga$ if necessary, we may assume that $\ga$ does not contain any root of $H(X)$ or nonzero $H_j(X)$. If $n = 0$ then clearly $L = \set{0}$ is as required. If $m >0$ then $(a, 0) \in A$ for every $a \in \ga$. So let us assume $n > 0$ and $m = 0$. Let $E \sub \set{0, 1, \ldots, n}$ be the subset such that $i \in E$ if and only if $X$ divides $H_i(X)$. Let
\[
G_1(X, Y) = \sum_{i \in E} H_i(X)Y^i, \,\,\,\,  G_2(X, Y) = \sum_{i \notin E} (H^*_i(X) + H_i(0)) Y^i.
\]
Note that $X$ also divides each $H^*_i(X)$. For any sufficiently large $\delta \in \Gamma$, $\vv(H_i(X))$ has a sufficiently large lower bound on $\go(0, \delta) \mi \set{0}$ for every $i \in E$; similarly for every $H^*_i(X)$. On the other hand, let $d_1, \ldots, d_k$ be the distinct roots of $G_2(0, Y) \in \VF(S)[Y]$ ($k = 0$ if $G_2(0, Y)$ is a nonzero constant) then, for any sufficiently large $\alpha \in \Gamma$, $\vv(G_2(0, b)) > \alpha$ only if $b \in \go(d_i, \epsilon)$ for some $i$. Therefore, if $\delta \in \Gamma$ is sufficiently large then for every $a \in \go(0, \delta) \mi \set{0}$ and every $b \notin \go(p_{\infty}, \epsilon) \cup \bigcup_i \go(d_i, \epsilon)$ we must have
\[
\vv(G^*(a, b) - G_2(0, b)) > \vv(G_2(0, b))
\]
and hence $G^*(a, b) \neq 0$. This concludes the basic case.

More generally, by compactness, $A \sub \VF^2$ is a union of finitely many subsets of the form $A_i \cap D_i$, where each $A_i$ is given by a $\VF$-sort equality as above. Since the lemma holds for each $A_i \cap D_i$, it holds for $A$ by Lemma~\ref{union:lim}.

For the case $m > 1$, let $A_i = \set{(b, \pr_i (\lbar a)) : (b, \lbar a) \in A}$ for each $i \leq m$ and $\lim_{\gb \mi \set{0} \rightarrow 0} A_i = L_i$. It is easy to see that
\[
\lim_{\gb \mi \set{0} \rightarrow 0} A \sub L_1 \times \cdots \times L_m
\]
and hence, as in the proof of Lemma~\ref{fin:lim:bounded}, there is a definable $L \sub L_1 \times \cdots \times L_m$ such that $\lim_{\gb \mi \set{0} \rightarrow 0} A = L$.
\end{proof}

\begin{defn}\label{defn:diff}
Let $f : \VF^n \fun \VF^m$ be a definable function. For any $\lbar a \in \VF^n$, we say that $f$ is \emph{differentiable at $\lbar a$} if there is a linear map $\lambda : \VF^n \fun \VF^m$ (of $\VF$-vector spaces) such that, for any $\epsilon \in \Gamma$, if $\lbar b \in \VF^n$ and $\vv(\lbar b)$ is sufficiently large then
\[
\vv(f(\lbar a + \lbar b) - f(\lbar a) - \lambda(\lbar b)) - \vv(\lbar b) > \epsilon.
\]
It is straightforward to check that if such a linear function $\lambda$ (a matrix with entries in $\VF$) exists then it is unique and hence may be called \emph{the derivative of $f$ at $\lbar a$}, which shall be denoted by $\der_{\lbar a} f$.

For each $1 \leq j \leq m$ let $f_j = \pr_j \circ f$. For any $\lbar a = (a_i, \lbar a_i) \in \VF^n$, if the derivative of the function $f_j \rest (\VF \times \set{\lbar a_i})$ at $a_i$ exists then we call it the \emph{$ij$th partial derivative of $f$ at $\lbar a$} and denote it by $\partial_{\lbar a}^{ij} f$.
\end{defn}

The classical differentiation rules, such as the product rule and the chain rule, hold with respect to this definition. Here we only check the chain rule:

\begin{lem}[The chain rule]\label{chain:rule}
Let $f : \VF^n \fun \VF^m$ be differentiable at $\lbar a \in \VF^n$ and $g : \VF^m \fun \VF^l$ differentiable at $f(\lbar a)$. Then $g \circ f$ is differentiable at $\lbar a$ and
\[
\der_{\lbar a} (g \circ f) = (\der_{f(\lbar a)} g) \times (\der_{\lbar a} f),
\]
where the righthand side is a product of matrices.
\end{lem}
\begin{proof}
Fix an $\epsilon \in \Gamma$. Since $\der_{\lbar a} f$ is a linear function, there is an $\alpha \in \Gamma$ such that, for every $\lbar b \in \VF^n$, $\vv(\der_{\lbar a} f(\lbar b)) - \vv(\lbar b) \geq \alpha$. Similarly there is a $\beta \in \Gamma$ such that, for every $\lbar b \in \VF^m$, $\vv(\der_{f(\lbar a)} g(\lbar b)) - \vv(\lbar b) \geq \beta$. Let $s: \VF^n \fun \VF^m$ be the function given by
\[
\lbar b \efun f(\lbar a + \lbar b) - f(\lbar a) - \der_{\lbar a} f(\lbar b).
\]
By assumption, for any $\lbar b \in \VF^n$ with $\vv(\lbar b)$ sufficiently large,
\[
\vv(\der_{f(\lbar a)} g(s(\lbar b))) \geq \vv (s(\lbar b)) + \beta > \vv(\lbar b) + (\epsilon - \beta) + \beta = \vv(\lbar b) + \epsilon.
\]
Therefore, if $\vv(\lbar b)$ is sufficiently large then either
\[
\vv (g(f(\lbar a + \lbar b)) - g(f(\lbar a)) - \der_{f(\lbar a)} g(\der_{\lbar a} f(\lbar b))) > \vv(\lbar b) + \epsilon
\]
or
\begin{align*}
     & \vv (g(f(\lbar a + \lbar b)) - g(f(\lbar a)) - \der_{f(\lbar a)} g(\der_{\lbar a} f(\lbar b)))\\
=  &   \vv (g(f(\lbar a + \lbar b)) - g(f(\lbar a)) - \der_{f(\lbar a)} g(\der_{\lbar a} f(\lbar b)) - \der_{f(\lbar a)} g(s(\lbar b)))\\
=    & \vv(g(f(\lbar a) + \der_{\lbar a} f(\lbar b) + s(\lbar b)) - g(f(\lbar a)) - \der_{f(\lbar a)} g(\der_{\lbar a} f(\lbar b) + s(\lbar b)))\\
>    & \vv(\der_{\lbar a} f(\lbar b) + s(\lbar b)) + \min \set{\beta, \epsilon - \alpha}\\
\geq & \vv(\lbar b) + \epsilon.
\end{align*}
In either case the lemma follows.
\end{proof}


\begin{lem}\label{partial:exist}
Let $f : \VF^n \fun \VF^m$ be a definable function. Then each partial derivative $\partial^{ij} f$ is defined almost everywhere.
\end{lem}
\begin{proof}
Let $\lbar a = (a_i, \lbar a_i) \in \VF^n$. Let $g^{ij}_{\lbar a} : \VF^{\times} \fun \VF$ be the $\lbar a$-definable function given by
\[
b \efun (f_j(a_i + b, \lbar a_i) - f_j(\lbar a)) / b,
\]
where $f_j = \pr_j \circ f$. By Lemma~\ref{VF:fun:reg:away}, for almost all $\lbar a \in \VF^n$ there is an $\lbar a$-definable open ball $\gb_{\lbar a}$ punctured at $0$ such that $\vv(g^{ij}_{\lbar a}(\gb_{\lbar a}))$ is bounded from below. By Lemma~\ref{lim:exists} and Lemma~\ref{fin:lim:bounded}, $\lim_{\gb_{\lbar a} \rightarrow 0} g^{ij}_{\lbar a} = \zeta(\lbar a)$ for some $\zeta(\lbar a) \in \VF$. The linear function is constructed in the usual way, taking $\zeta(\lbar a)$ as the slope.
\end{proof}

\begin{cor}\label{diff:almost:every}
Let $f : \VF^n \fun \VF^m$ be a definable function. Then $f$ is continuously partially differentiable almost everywhere.
\end{cor}
\begin{proof}
This is immediate by Lemma~\ref{partial:exist} and Lemma~\ref{fun:almost:cont}.
\end{proof}


We would like to differentiate functions between arbitrary definable subsets. The simplest way to do this to be ``forgetful'' about the $\RV$-coordinates. Let $f : \VF^n \times \RV^m \fun \VF^{n'} \times \RV^{m'}$ be a definable function. For each $\lbar t \in \RV^m$ let $U_{\lbar t} = \prv(f(\VF^n \times \{\lbar t\}))$. For every $\lbar s \in U_{\lbar t}$ let $f_{\lbar t, \lbar s}$ be the function on $\{\lbar a : \prv(f(\lbar a, \lbar t)) = \lbar s \}$ given by $\lbar a \efun \pvf(f(\lbar a, \lbar t))$. Note that, by compactness, there is an $\lbar s \in U_{\lbar t}$ such that $\dim_{\VF}(\dom(f_{\lbar t, \lbar s})) = n$ and hence, by Lemma~\ref{full:dim:open:poly}, $\dom(f_{\lbar t, \lbar s})$ contains an open polydisc. For such an $\lbar s$ and each $\lbar a \in \dom(f_{\lbar t, \lbar s})$ we define the \emph{$ij$th partial derivative of $f$ at $(\lbar a, \lbar t)$} to be the $ij$th partial derivative of $f_{\lbar t, \lbar s}$ at $\lbar a$. It follows from Corollary~\ref{diff:almost:every} and compactness that every partial derivative of $f$ is defined almost everywhere.

\begin{defn}
If $n = n'$ and all the partial derivatives exist at a point $(\lbar a, \lbar t)$ then the \emph{Jacobian of $f$ at $(\lbar a, \lbar t)$} is defined in the usual way, that is, the determinant of the Jacobian matrix, and is denoted by $\jcb_{\VF} f(\lbar a, \lbar t)$.
\end{defn}

\begin{lem}\label{special:tran:vol:pre}
For any special bijection $T : A \fun A^{\sharp}$, the Jacobians of $T$ and $T^{-1}$ are equal to $1$ almost everywhere. If $A$ is a nondegenerate $\RV$-pullback then they are equal to $1$ everywhere.
\end{lem}
\begin{proof}
We may assume that the length of $T$ is $1$. Then this is clear if we apply the proof of Lemma~\ref{partial:exist} to (additive) translation and canonical bijection (or its inverse).
\end{proof}

\begin{lem}\label{jcb:chain}
Let $f : A \fun B$ and $g : B \fun C$ be definable functions. Then for any $\lbar x \in A$,
\[
\jcb_{\VF} ( g \circ f)(\lbar x) = \jcb_{\VF} g(f(\lbar x)) \cdot \jcb_{\VF} f(\lbar x),
\]
if both sides are defined.
\end{lem}
\begin{proof}
This is immediate by the chain rule.
\end{proof}


Next we describe the second approach to defining the Jacobian in $\VF$. Let $f : \VF^n \fun \VF^n$ be a definable function, which in general is not a rational map. Let $D \sub \VF^{2n}$ be the Zariski closure of the graph of $f$. By Proposition~\ref{dim:vf:same:zar} the dimension of $D$ is $n$ and hence $\pr_{\leq n} \rest D$ is finite-to-one. Let $D_1 = \pr_{\leq n}(D) = \VF^n$ and $D_2 = \pr_{> n}(D)$. For almost all $(\lbar a_1, \lbar a_2) = \lbar a \in D$, $\pr_{\leq n}$ and $\pr_{> n}$ induce surjective linear maps of the tangent spaces (see~\cite[Lemma~2, p.~141]{Shafarevich:77}):
\[
\der_{\lbar a} \pr_{\leq n} : T_{\lbar a}(D) \fun T_{\lbar a_1}(D_1), \quad \der_{\lbar a} \pr_{> n}: T_{\lbar a}(D) \fun T_{\lbar a_2}(D_2).
\]
Since the dimension of $D_1$ is also $n$, we see that $\der_{\lbar a} \pr_{\leq n}$ is an isomorphism of the tangent spaces for almost all $(\lbar a_1, \lbar a_2) = \lbar a \in D$ and hence the composition
\[
(\der_{\lbar a} \pr_{> n}) \circ (\der_{\lbar a} \pr_{\leq n})^{-1} : T_{\lbar a_1}(D_1) \fun T_{\lbar a_2}(D_2)
\]
is defined and is given by an $n \times n$ matrix $\lambda_{\lbar a}$ with entries in $\VF$ (not necessarily invertible). Suppose $f(\lbar a_1) = \lbar a_2$. Then $\lambda_{\lbar a}$ satisfies the defining property in Definition~\ref{defn:diff} and hence $\det \lambda_{\lbar a} = \jcb_{\VF} f(\lbar a_1)$. It is clear that this equality holds for almost all $\lbar a_1 \in \VF^n$. Note that the construction can be carried out even if $f$ is a partial function, as long as $\dim_{\VF}(\dom(f)) = n$.

Now the Jacobian in $\RV$ may be defined almost identically as above. But for clarity we shall repeat the whole procedure. Let $(U, f), (V, g) \in \RV[n, \cdot]$. Set $A = f(U) \cap (\RV^{\times})^n$ and $B = g(V) \cap (\RV^{\times})^n$.
\begin{defn}
An \emph{essential isomorphism} between $(U, f)$ and $(V, g)$ is an isomorphism between $(f^{-1}(A), f \rest f^{-1}(A))$ and $(g^{-1}(B), g \rest g^{-1}(B))$.
\end{defn}
Let $F : (U, f) \fun (V, g)$ be an essential isomorphism. Note that if $A \neq \0$ then a lift of $F$ is defined almost everywhere on $\bb L(U, f)$. Actually, since the parts $f(U) \mi A$ and $g(V) \mi B$ will not concern us, we may assume $f^{-1}(A) = U$ and $g^{-1}(B) = V$. We also assume that $A, B$ are of $\RV$-dimension $n$. Set
\[
C = \set{(f(\lbar u), g(F(\lbar u))) : \lbar u \in U } \sub A \times B.
\]
Note that, since $F$ is an isomorphism, both $\pr_{\leq n} \rest C$ and $\pr_{> n} \rest C$ are finite-to-one. We first consider the simple situation $A, B\sub (\K^{\times})^n$. By Remark~\ref{def:subseti:in:K}, $A, B$ are unions of locally closed subsets (in the sense of Zariski topology). We may assume that $A, B, C$ are varieties. Clearly the dimensions of $A, B, C$ are all $n$. Since the projections $\pi_A, \pi_B$ of $C$ to $A$ and $B$ are dominant rational maps, for almost all $(f(\lbar u), g(F(\lbar u))) = \lbar c \in C$ (that is, outside of a closed subset of dimension $< n$), $\pi_A, \pi_B$ induce isomorphisms of the tangent spaces:
\[
\der_{\lbar c} \pi_A : T_{\lbar c}(C) \fun T_{\pi_A(\lbar c)}(A), \quad \der_{\lbar c} \pi_B: T_{\lbar c}(C) \fun T_{\pi_B(\lbar c)}(B).
\]
Therefore the composition
\[
(\der_{\lbar c} \pi_B) \circ (\der_{\lbar c} \pi_A)^{-1} : T_{\pi_A(\lbar c)}(A) \fun T_{\pi_B(\lbar c)}(B)
\]
is defined and is given by an invertible $n \times n$ matrix $\lambda_{\lbar u}$ with entries in $\K$. The determinant of $\lambda_{\lbar u}$, denoted by $\jcb_{\K} F(f(\lbar u), \lbar u)$, is the \emph{Jacobian of $F$ at $\lbar u$}, which is a $\lbar u$-definable element in $\K^{\times}$. Note that $\jcb_{\K} F$ is defined almost everywhere in $A$, that is, the subset of those $f(\lbar u) \in A$ such that $\jcb_{\K} F(f(\lbar u),\lbar u)$ is not defined is of dimension $< n$.

In general, if $(f(\lbar u), g(F(\lbar u))) \in C$ is contained in a multiplicative coset $O$ of $(\K^{\times})^{2n}$ then we may translate $A, B$ coordinate-wise by $f(\lbar u), g(F(\lbar u))$ respectively so that $O$ is mapped into $(\K^{\times})^{2n}$. Let $(U, f')$, $(V, g')$ be the induced objects and $F'$ the induced isomorphism on $f'^{-1}((\K^{\times})^{n})$.

\begin{defn}
The \emph{Jacobian $\jcb_{\RV} F(f(\lbar u),\lbar u)$ of $F$ at $\lbar u$} is a $\lbar u$-definable element in $\RV^{\times}$ given by
\[
(\Pi f(\lbar u))^{-1} (\Pi g(F(\lbar u))) \jcb_{\K} F'(1, \ldots, 1)
\]
if it exists, where $\Pi (t_1, \ldots, t_n) = t_1 \times \cdots \times t_n$.
\end{defn}

By Lemma~\ref{rv:dim:gamma:coset} and compactness, the subset of those $f(\lbar u) \in A$ such that $\jcb_{\RV} F(f(\lbar u),\lbar u)$ is defined is not empty and the subset of those $f(\lbar u) \in A$ such that $\jcb_{\RV} F(f(\lbar u),\lbar u)$ is not defined is of dimension $< n$. Symmetrically this is also true for $B$.

We may further coarsen the data and define the \emph{$\Gamma$-Jacobian}
\[
\jcb_{\Gamma} F(f(\lbar u),\lbar u) = \Sigma (\vrv \circ g \circ F)(\lbar u) - \Sigma (\vrv \circ f)(\lbar u),
\]
where $\Sigma (\gamma_1, \ldots, \gamma_n) = \gamma_1 + \cdots + \gamma_n$. Obviously this always exists and
\[
\vrv(\jcb_{\RV} F(f(\lbar u),\lbar u)) = \jcb_{\Gamma} F(f(\lbar u),\lbar u).
\]

Note that the chain rule clearly holds for both $\jcb_{\RV}$ and $\jcb_{\Gamma}$ whenever the things involved are defined.

For the rest of this section we do not need to assume that $A$, $B$ are of $\RV$-dimension $n$.

\begin{lem}\label{RV:iso:class:lifted:jcb}
Let $F^{\uparrow} : \bb L(U, f) \fun \bb L (V, g)$ be a lift of $F$. Then for every $f(\lbar u) \in A$ outside of a definable subset of $A$ of dimension $<n$ and almost all $(\lbar a, \lbar u) \in \rv^{-1}(f(\lbar u),\lbar u)$,
\[
\rv(\jcb_{\VF} F^{\uparrow}(\lbar a, \lbar u)) = \jcb_{\RV} F(f(\lbar u), \lbar u).
\]
Also, for almost all $(\lbar a, \lbar u) \in \bb L(U, f)$,
\[
\vv(\jcb_{\VF} F^{\uparrow}(\lbar a, \lbar u)) = \jcb_{\Gamma} F(f(\lbar u), \lbar u).
\]
\end{lem}
\begin{proof}
Without loss of generality we may assume $\dim_{\RV}(A) = n$. Also, by Lemma~\ref{jcb:chain} and compactness, we may assume $A, B \sub (\K^{\times})^{n}$. For almost all $(\lbar a, \lbar u) \in \bb L(U, f)$, $\jcb_{\VF} F^{\uparrow}(\lbar a, \lbar u)$ may be obtained by running the construction described above with respect to  $\rv^{-1}(A), \rv^{-1}(B), \rv^{-1}(C)$ and the projection maps. For almost all $f(\lbar u) \in A$ this construction modulo the maximal ideal agrees with the construction that yields $\jcb_{\RV} F(f(\lbar u), \lbar u)$. The second assertion follows from Lemma~\ref{jcb:chain}.
\end{proof}

Let $a, b\in \OO$ be definable units. Set $\rv(a) = t$ and $\rv(b) = s$. Clearly for any definable unit $c \in \OO$ there is a definable bijection $f : \rv^{-1}(t) \fun \rv^{-1}(s)$ such that $\der_{x} f = c$ for all $x \in \rv^{-1}(t)$. This simple observation is used in the following analogue of Theorem~\ref{RV:iso:class:lifted}, where we need to assume that $f, g$ are finite-to-one, that is, $(U, f), (V, g) \in \RV[n]$ (for otherwise we may not have definable points in $\VF$ to work with).

\begin{thm}\label{iso:lifted:vol}
Suppose that $S$ is $(\VF, \Gamma)$-generated and $f, g$ are finite-to-one. Let $\omega : U \fun \RV$ be a definable function such that
\begin{enumerate}
  \item $\omega(\lbar u) = \jcb_{\RV} F(f(\lbar u), \lbar u)$ for every $\lbar u \in U$ outside of a definable subset of $U$ of dimension $<n$,
  \item $\vrv(\omega(\lbar u)) = \jcb_{\Gamma} F(f(\lbar u), \lbar u)$ for \emph{every} $\lbar u \in U$.
\end{enumerate}
Then there is a lift $F^{\uparrow} : \bb L(U, f) \fun \bb L (V, g)$ of $F$ such that for almost all $(\lbar a, \lbar u) \in \bb L(U, f)$,
\[
\rv(\jcb_{\VF} F^{\uparrow}(\lbar a, \lbar u)) = \omega(\lbar u).
\]
\end{thm}
\begin{proof}
As in the proof of Lemma~\ref{RV:iso:class:lifted:jcb} we may assume $A, B \sub (\K^{\times})^{n}$ and hence $\vrv \circ \omega$ is the zero function. By Theorem~\ref{RV:iso:class:lifted} and Lemma~\ref{RV:iso:class:lifted:jcb} we are reduced to showing this for a definable subset $A_1 \sub A$ of $\RV$-dimension $< n$. We do induction on $\dim_{\RV}(A_1)$. For the base case, since $A_1$ is finite, by Lemma~\ref{algebraic:balls:definable:centers} the $\rv$-balls involved have centers, then it is easy to see that we may apply the simple observation above in one of the coordinates and use additive translation in the other coordinates.

We proceed to the inductive step. Let $f^{-1}(A_1) = U_1$, $F(U_1) = V_1$, and $B_1 = (g \circ F)(U_1)$. Since $\dim_{\RV}(A_1) = k < n$, without loss of generality, we may assume over a definable finite partition of $A_1$ that both $\pr_{\leq k} \rest A_1$ and $\pr_{\leq k} \rest B_1$ are finite-to-one. Let
\[
f_1 : U_1 \fun \pr_{\leq k}(A_1), \quad g_1 : V_1 \fun \pr_{\leq k}(B_1), \quad F_1 : (U_1, f_1) \fun (V_1, g_1)
\]
be the naturally induced definable functions and
\[
C_1 = \{(f_1(\lbar u), g_1(F_1(\lbar u))) : \lbar u \in U_1 \} \sub \pr_{\leq k}(A_1) \times \pr_{\leq k}(B_1).
\]
Clearly both $\pr_{\leq k} \rest C_1$ and $\pr_{> k}\rest C_1$ are finite-to-one and hence, by Theorem~\ref{RV:iso:class:lifted} and Lemma~\ref{RV:iso:class:lifted:jcb} again, there is a definable subset $A_2 \sub \pr_{\leq k}(A_1)$ and a lift $F_1^{\uparrow}$ of $F_1$ such that $\dim_{\RV}(\pr_{\leq k}(A_1) \mi A_2) < k$ and for all $f_1(\lbar u) \in A_2$ and almost all $(\lbar a, \lbar u) \in \rv^{-1}(f_1(\lbar u),\lbar u)$,
\[
\rv(\jcb_{\VF} F_1^{\uparrow}(\lbar a, \lbar u)) = \jcb_{\RV} F_1(f_1(\lbar u), \lbar u).
\]
Let $U_2 =  (\pr_{\leq k} \circ f)^{-1} (A_2)$. By the inductive hypothesis there is a lift of $F \rest (U_1 \mi U_2)$ as desired.

We construct a lift $F_2^{\uparrow}$ of $F \rest U_2$ as follows. Let $\lbar t \in A_2$ and $U_{\lbar t} = f^{-1}(\fib(A_1, \lbar t))$. For any $\lbar a \in \rv^{-1}(\lbar t)$ we have $\lbar a$-definable centers
\[
h_{\lbar a} : \fib(A_1, \lbar t) \cup \omega(U_{\lbar t}) \fun \OO \mi \MM.
\]
For any $(\lbar a, \lbar u) \in \rv^{-1}(\lbar t,\lbar u)$, using the centers provided by $h_{\lbar a}$ as above, we may construct an $(\lbar a, \lbar u)$-definable bijection
\[
F_{\lbar a, \lbar u} : \rv^{-1}((\pr_{>k} \circ f)(\lbar u)) \fun \rv^{-1}((\pr_{>k} \circ g \circ F)(\lbar u))
\]
such that, for any $\lbar b \in \dom(F_{\lbar a, \lbar u})$,
\[
\jcb_{\VF} F_{\lbar a, \lbar u} (\lbar b) = (\jcb_{\VF} F_1^{\uparrow}(\lbar a, \lbar u))^{-1} h_{\lbar a}(\lbar u)
\]
if the righthand side is defined; otherwise let $F_{\lbar a, \lbar u}$ be any $(\lbar a, \lbar u)$-definable bijection. Now let $F_2^{\uparrow}$ be the lift of $F \rest U_2$ given by
\[
(\lbar a, \lbar b, \lbar u) \efun (\lbar a, F_{\lbar a, \lbar u} (\lbar a, \lbar b), \lbar u) \efun (F_1^{\uparrow}(\lbar a, \lbar u), F_{\lbar a, \lbar u} (\lbar a, \lbar b)).
\]
Multiplying the Jacobians of the two components (Lemma~\ref{jcb:chain}), we see that $F_2^{\uparrow}$ is as desired.
\end{proof}

%

\section{Categories with volume forms}\label{section:cat:vol}

In this section we shall assume that the substructure $S$ is $(\VF, \Gamma)$-generated.

We shall define finer categories of definable subsets with the notion of the Jacobian factored in. This will make the homomorphisms between various Grothendieck groups compatible with the Jacobian transformation, as in the classical integration theory.

\begin{defn}[$\VF$-categories with volume forms]\label{defn:mVF:cat:vol}
First set $\mVF[0, \cdot] = \VF[0, \cdot]$. Suppose $k > 0$. An object in the category $\mVF[k, \cdot]$ is a definable pair $(A, \omega)$, where $\pvf(A) \sub \VF^k$ and $\omega : A \fun \RV^{\times}$ is a function. The latter is understood as a \emph{definable $\RV$-volume form} on $A$, or simply a volume form on $A$. A \emph{morphism} between two objects $(A, \omega)$, $(A', \omega')$ is a definable \emph{essential bijection} $F : A \fun A'$, that is, a bijection that is defined outside of definable subsets of $A$, $A'$ of $\VF$-dimension $< k$, such that for every $\lbar x \in \dom(F)$,
\[
\omega(\lbar x) = \omega'(F(\lbar x)) \cdot \rv(\jcb_{\VF} F(\lbar x)).
\]
We also say that such an $F$ is an \emph{$\RV$-measure-preserving map}, or simply measure-preserving map.

An object in the category $\mgVF[k, \cdot]$ is a pair $(A, \omega)$, where $A \in \VF[k, \cdot]$ and $\omega : A \fun \Gamma$ a definable function. The latter is understood as a \emph{definable $\Gamma$-volume form} on $A$. A morphism between two objects $(A, \omega), (A', \omega')$ is a definable essential bijection $F : A \fun A'$ such that for every $\lbar x \in \dom(F)$,
\[
\omega(\lbar x) = \omega'(F(\lbar x)) + \vv(\jcb_{\VF} F(\lbar x)).
\]
We also say that such an $F$ is a \emph{$\Gamma$-measure-preserving map}.

The category $\VF_1[k, \cdot]$ is the full subcategory of $\mVF[k, \cdot]$ such that $(A, \omega) \in \VF_1[k, \cdot]$ if and only of $\omega = 1$. The category $\VF_0[k, \cdot]$ is the full subcategory of $\mVF[k, \cdot]$ such that $(A, \omega) \in \VF_0[k, \cdot]$ if and only of $\omega = 0$.

The category $\mVF[k]$ is the full subcategory of $\mVF[k, \cdot]$ such that $(A, \omega) \in \mVF[k]$ if and only of $A \in \VF[k]$; similarly for the categories $\mgVF[k]$, $\VF_1[k]$, $\VF_{0}[k]$.

The category $\mVF_*[\cdot]$ is defined to be the \emph{direct sums} (coproducts) of the corresponding categories; similarly for the other ones.
\end{defn}

Note that, for conceptual simplicity, we have allowed redundant objects in these categories. For example, if $(A, \omega) \in \mVF[k, \cdot]$ with $\dim_{\VF}(A) < k$ then $(A, \omega)$ is isomorphic to the empty object. Also, given how each $\mVF[k, \cdot]$ is defined, $\mVF_*[\cdot]$ is actually just the union of the corresponding categories.

\begin{rem}\label{mor:equi}
Any two morphisms in $\mVF[k, \cdot]$ that agree almost everywhere may be naturally identified. It is conceptually more ``correct'' to define a morphism in $\mVF[k, \cdot]$ as such an equivalence class, although in practice it is more convenient to work with a representative. The ``equivalence class'' point of view is required when it comes to defining the Grothendieck semigroup. Consequently, since the Jacobian of the identity map is equal to $1$ almost everywhere, by Lemma~\ref{jcb:chain}, every morphism is actually an isomorphism. This is very similar to birational maps in algebraical geometry. Below by a ``morphism'' we shall mean either an equivalence class or a representative of the class, depending on the context.
\end{rem}


\begin{defn}[$\RV$-categories with volume forms]\label{defn:RV:cat:vol}
First set $\mRV[0] = \RV[0]$. Suppose $k > 0$. An object of the category $\mRV[k]$ is a definable triple $(U,f, \omega)$, where $(U, f) \in \RV[k]$ and $\omega: U \fun \RV^{\times}$ is a function, which is understood as a \emph{volume form} on $(U, f)$. A \emph{morphism} between two objects $(U,f, \omega)$, $(U', f', \omega')$ is an essential isomorphism $F : (U, f) \fun (U', f')$ such that
\begin{enumerate}
  \item $\omega(\lbar u) = \omega'(F(\lbar u)) \cdot \jcb_{\RV} F(f(\lbar u), \lbar u)$ for every $\lbar u \in \dom(F)$ outside of a definable subset of $\dom(F)$ of dimension $< k$,
  \item $\vrv(\omega(\lbar u)) = (\vrv \circ \omega' \circ F)(\lbar u) + \jcb_{\Gamma} F(f(\lbar u), \lbar u)$ for every $\lbar u \in \dom(F)$.
\end{enumerate}
It is easily seen from the definitions of $\jcb_{\RV}$ and $\jcb_{\Gamma}$ that every morphism here is actually an isomorphism.

The categories $\mgRV[k]$, $\RV_1[k]$, $\RV_0[k]$ are similar to the corresponding $\VF$-categories.

The categories $\mRV[\leq k]$, $\mRV[*]$ are defined to be the \emph{direct sums} (coproducts) of the corresponding categories; similarly for the other ones.
\end{defn}

Note that, as in the $\VF$-categories with volume forms, we have allowed redundant objects in the $\RV$-categories with volume forms. For example, for an object $(\mathbf{U}, \omega)$, if $\bb L \mathbf{U}$ is strictly degenerate then $(\mathbf{U}, \omega)$ is isomorphic to the empty object.

For any $(\mathbf{U}, \omega) \in \mRV[k]$, let $\bb L \omega$ be the function on $\bb L \mathbf{U}$ naturally induced by $\omega$. The \emph{lift} of $(\mathbf{U}, \omega)$ is the object $\bb L(\mathbf{U}, \omega) = (\bb L \mathbf{U}, \bb L \omega) \in \mVF[k]$.

For each $(A, \omega) \in \mVF[k]$ let $A_{\omega} = \set{(\lbar a, \omega(\lbar a)) : \lbar a \in A)}$. The function $\omega$ induces naturally a function on $A_{\omega}$, which will also be denoted by $\omega$ for simplicity. Clearly $(A, \omega)$ and $(A_{\omega}, \omega)$ are isomorphic.

\begin{thm}\label{L:measure:surjective}
Every object $(A, \omega)$ in $\mVF[k]$ is isomorphic to another object $\bb L(\mathbf{U}, \pi)$ in $\mVF[k]$, where $(\mathbf{U}, \pi) \in \mRV[k]$; similarly for other pairs of corresponding categories.
\end{thm}
\begin{proof}
By Corollary~\ref{all:subsets:rvproduct} there is a special bijection $T : A_{\omega} \fun A_1$ with $A_1$ an $\RV$-pullback such that $(\rv(A_1), \pr_{\leq k}) \in \RV[k]$. Let $\omega_1 = \omega \circ T^{-1}$. So $\omega_1$ is constant on every $\rv$-polydisc. By Corollary~\ref{diff:almost:every} and Lemma~\ref{special:tran:vol:pre}, $(A, \omega)$ and $(A_1, \omega_1)$ are isomorphic. Let $\pi : \rv(A_1) \fun \RV$ be the function naturally induced by $\omega_1$. Then $(\rv(A_1), \pr_{\leq k}, \pi)$ is as required.

The arguments for the other cases are essentially the same.
\end{proof}

\begin{thm}\label{L:measure:class:lift}
Let $F : (\mathbf{U}, \omega) \fun (\mathbf{U}', \omega')$ be a $\mRV[k]$-isomorphism. Then there exists a measuring-preserving lift $F^{\uparrow} : \bb L(\mathbf{U}, \omega) \fun \bb L (\mathbf{U}', \omega')$ of $F$.
\end{thm}
\begin{proof}
Let $\omega^* : \dom(F) \fun \RV$ be the function given by $\lbar u \efun \omega(\lbar u) / \omega'(F(\lbar u))$. By Theorem~\ref{iso:lifted:vol}, there is a lift $F^{\uparrow} : \bb L \mathbf{U} \fun \bb L \mathbf{U}'$ such that $\rv(\jcb_{\VF} F^{\uparrow}(\lbar a, \lbar u)) = \omega^*(\lbar u)$ for almost all $(\lbar a, \lbar u) \in \bb L \mathbf{U}$, that is, $F^{\uparrow}$ is a $\mVF[k]$-isomorphism between $\bb L(\mathbf{U}, \omega)$ and $\bb L(\mathbf{U}', \omega')$.
\end{proof}

\begin{cor}\label{L:measure:semigroup:hom}
The map $\mathbb{L}$ induces surjective homomorphisms between the various
Grothendieck semigroups associated with the categories with volume forms, for example:
\[
\gsk \mRV[k] \fun \gsk \mVF[k], \quad \gsk \mgRV[k] \fun \gsk \mgVF[k].
\]
\end{cor}

As mentioned in Step~3 in the introduction, various classical properties, in particular, special cases of Fubini's theorem and a change of variables formula, can already be verified for the inversions of the homomorphisms in Corollary~\ref{L:measure:semigroup:hom} and hence we may complete the Hrushovski-Kazhdan construction of motivic integration right here. However, we choose to postpone this until we have achieved a more satisfying theory by putting forward a canonical description of the kernels of these homomorphisms in a sequel.

\end{document}